\documentclass[a4paper,oneside,11pt]{article}

\usepackage[utf8]{inputenc} 
\usepackage[T1]{fontenc} 
\usepackage{textcomp} 
\usepackage[english]{babel}
\usepackage[autolanguage]{numprint} 
\usepackage{hyperref} 
\usepackage{graphicx} 
\usepackage{verbatim}
\usepackage[all]{xy}
\usepackage{dsfont}
\usepackage{amsmath,amssymb,amsthm} 
\usepackage{lmodern}
\usepackage{ mathrsfs }

\renewcommand\epsilon\varespilon 
\renewcommand\phi\varphi

\newcommand\NN{\mathbb{N}}

\newcommand\RR{\mathbb{R}}
\newcommand\CC{\mathbb{C}} 

\newcommand\PP{\mathbb{P}}
\newcommand\EE{\mathbb{E}}

\newcommand\MM{\mathbb{M}}
\newcommand\tr{\mathrm{tr}}
\newcommand\Car{\mathds{1}}
\newcommand\bornesup{\limsup_{N\to +\infty} \frac{1}{N^{1+2/p}} \log}

\newcommand\bornesupalpha{\limsup_{N\to +\infty} N^{-\alpha\left(\frac{1}{2}+\frac{1}{p}\right) } \log}

\newcommand \bornexpoalpha{\lim_{N\to +\infty} N^{-\alpha\left(\frac{1}{2}+\frac{1}{p}\right) }  \log }
\newcommand\veps{\varepsilon}

\newcommand\M{\mathcal{M}}
\newcommand\Card{\mathrm{Card}}
\newcommand\supp{\mathrm{supp}}
\newcommand\sg{\mathrm{sg}}

\swapnumbers 
\theoremstyle{definition} 
\newtheorem{Def}{Definition}[section] 
\theoremstyle{plain} 
\newtheorem{Pro}[Def]{Proposition} 
\newtheorem{Lem}[Def]{Lemma} 
\newtheorem{The}[Def]{Theorem} 
\newtheorem{Cor}[Def]{Corollary}
 
\theoremstyle{remark}

\newtheorem{Rem}[Def]{Remark}

\title{On the large deviations of traces of random matrices} 
\author{Fanny Augeri } 
\date{\today}

\begin{document}

\maketitle

\section*{Abstract}
We present large deviations principles for the moments of the empirical spectral measure of Wigner matrices and empirical measure of $\beta$-ensembles in three cases : the case of Wigner matrices without Gaussian tails, that is Wigner matrices whose entries have tail distributions decreasing as $e^{-ct^{\alpha}}$, for some constant $c>0$ and with $\alpha \in (0,2)$, the case of Gaussian Wigner matrices, and the case of $\beta$-ensembles associated with a convex  potential with polynomial growth.

\section{Introduction and main results}

The study of the traces of random matrices is now a classical tool to understand the behavior of their spectrum. From the original proof of Wigner's theorem by the moments method (see \cite{Wigner}), to the universality results at the edge of Hermitian or covariance random matrices (see for example \cite{Soshnikovuniv}, \cite{Sodin}), 'Wigner traces method' has proven extremely effective in the macroscopic, as well as the microscopic study of the spectrum of random matrices.

Starting from Wigner's theorem, which asserts that for a standard Wigner matrix whose entries are centered and have finite moments, the moments of the empirical spectral measure, or equivalently the normalized traces, converges almost surely to $0$, for odd moments, and to the Catalan numbers, for even moments, one can ask about the deviations of these moments around their respective limit value. 

The fluctuations of the traces of random matrices have been extensively studied, usually as a first step to get the fluctuations of the linear statistics of the eigenvalues. Originally proven in the context of Wishart matrices in \cite{Jonsson}, a central limit theorem for the moments of the empirical spectral measure of standard Wigner matrices can be found in \cite[Theorem 2.1.31]{Guionnet}, following Jonsson's strategy of using the moments method and combinatorial techniques. Due to the repulsion of the eigenvalues, one has to multiply by a factor $N$ - instead of $\sqrt{N}$ in the case of independent variables - to see the fluctuations of the centered moments. The development of the combinatorial approach culminated in \cite{Sinai1}, \cite{Sinai2}, in which the authors show a CLT for the $p^{\text{th}}$ moment with $p$ growing with $N$, $p\ll N^{2/3}$, as well as multivariate version of the CLT for moments, in the case of standard Wigner matrices with symmetric and sub-Gaussian entries.

Regarding the deviations of the moments of the empirical spectral measure, we know from \cite[section 3.1]{Ledouxbanach}, that the $p$-Schatten norm of Gaussian Hermitian or symmetric matrices is sub-Gaussian. Still in the Gaussian case, the estimates of moments of Gaussian chaos of \cite{Latala} can also provide some concentration inequalities for the moments of the empirical spectral measure. Concentration inequalities for truncated traces of convex perturbation of the GUE multi-matrix model can be found in \cite{GuionnetFlour}.
More generally, we know by \cite{Meckes}, that if the entries of $X$ are bounded or satisfies some logarithmic Sobolev inequalities, then the normalized traces of powers of $X$, say $\tr_N (X/\sqrt{N})^p$, satisfies a concentration inequality with speed $N^{1+2/p}$. This gives an indication, at least in the case where the entries of $X$ are bounded or satisfy logarithmic Sobolev inequalities, of the speed of the large deviations of the moments of the empirical spectral measure around the Catalan numbers.

Note that since the map which associates to a probability measure on $\RR$, its $p^{\text{th}}$ moment is not continuous for the weak topology, one cannot derive, by a contraction principle, large deviations principles for the $p^{\text{th}}$ moment of the empirical spectral measure, from the already known large deviations principles for the empirical spectral measure, like in the case of the GUE or GOE due to \cite{BenArous}, or in the case of Wigner matrices without Gaussian tails due to \cite{Bordenave}.

Moderate deviations of certain traces of convex perturbation of the GUE multi-matrix model have been investigated in \cite{Eichel}. 
In the case where the entries are not centered, some results of large deviations for the moments of the empirical spectral measure are known.  In the case of symmetric Bernoulli matrices, we know by \cite[Theorem 1.5]{Eichel2} that the centered traces satisfy moderate deviations principles with an explicit rate function. A large deviations principle for the traces of Bernoulli matrices is derived in \cite[Theorem 4.1]{Chatterjee2}, as a consequence of the large deviations principle of Erdös-Renyi graphs with parameter $p$ independent of $N$, with respect to the cut metric.

\subsection{Main results}
The aim of this paper is to derive large deviations principles for the moments of the empirical (spectral) measure in three cases : the case of $\beta$-ensembles for convex potential with polynomial growth in section \ref{betaens}, the case of Gaussian Wigner matrices in section \ref{Gaussian}, and the case of Wigner matrices without Gaussian tails in section \ref{noGaussian}.

We recall that a sequence of random variables $(Z_n)_{n\in \NN}$ taking value in some topological space $\mathcal{X}$ equipped with the Borel $\sigma$-field $\mathcal{B}$, follows a large deviations principle (LDP) with speed $\upsilon : \NN\to \NN$, and rate function $J : \mathcal{X} \to [0, +\infty]$, if $J$ is lower semicontinuous and $\upsilon$ increases to infinity and for all $B\in \mathcal{B}$,
$$- \inf_{B^{\circ}}J \leq \liminf_{n\to +\infty} \frac{1}{\upsilon(n)} \log \PP\left(Z_n \in B\right) \leq \limsup_{n\to +\infty} \frac{1}{\upsilon(n)}\log\PP\left(Z_n \in B\right) \leq -\inf_{\overline{B}  } J,$$
where $B^{\circ}$ denotes the interior of $B$ and $\overline{B}$ the closure of $B$. We recall that $J$ is lower semicontinuous if its $t$-level sets $\{ x \in \mathcal{X} : J(x) \leq t \}$ are closed, for any $t\in [0,+\infty)$. Furthermore, if all the level sets are compact, then we say that $J$ is a good rate function.

We define the $\beta$-ensemble associated with the potential $V$ as the following probability measure on $\RR^N$,
\begin{equation} \label{defbeta} d\PP_{V, \beta}^N =\frac{1}{Z^N_{V}} \prod_{i<j}\left|\lambda_i- \lambda_j\right|^{\beta} e^{-N\sum_{i=1}^N V\left(\lambda_i\right)}\prod_{i=1}^N d\lambda_i,\end{equation}
where $Z^N_{V,\beta}$ is the partition function, that is,
\begin{equation} \label{partition func} Z^N_{V, \beta} = \int  \prod_{i<j}\left|\lambda_i- \lambda_j\right|^{\beta} e^{-N\sum_{i=1}^N V\left(\lambda_i\right)}\prod_{i=1}^N d\lambda_i.\end{equation}
To make sense of $\PP_{V, \beta}^N$, it is usually assumed that $V$ is a continuous function such that there is some $\beta'>1$, $\beta'\geq \beta$, such that 
\begin{equation} \label{assumpgene} \liminf_{|x| \to +\infty} \frac{V(x)}{\beta' \log |x|} >1.\end{equation}
It is known (see \cite[Theorem 2.6.1]{Guionnet} or \cite{BenArous}),  that the empirical measure
$$ L_N = \frac{1}{N} \sum_{i=1}^N \delta_{\lambda_i },$$
follows, under $\PP_{V, \beta}^N$, a LDP with respect to the weak topology, with speed $N^2$, and good rate function $I_{\beta}^V$. Furthermore, $I_{\beta}^V$ achieves its minimum at a unique probability measure $\sigma_{\beta}^V$, called the equilibrium measure, which is compactly supported (see \cite[Lemma 2.6.2]{Guionnet}).

In the case of $\beta$-ensembles associated with a convex potential with polynomial growth, the following holds.

\begin{The}\label{ldp1}
Let $\alpha\geq 2$ and $\beta>0$. Let
\begin{equation} \label{assumpV} \forall x\in \RR, \ V(x) = b|x|^{\alpha} + w(x),\end{equation}
where $w$ is a continuous convex function such that $w(x) = o_{\pm \infty}(|x|^{\alpha})$. 
Let $p\in \NN$, $p> \alpha$. For any $\lambda_1,...,\lambda_N \in \RR^N$, we denote by $m_{p,N}$,
$$m_{p,N} = \frac{1}{N} \sum_{i=1}^N \lambda_i^p.$$
Under $\PP^N_{V, \beta}$, the sequence $\left(m_{p, N}\right)_{N\geq 1}$ satisfies a large deviations principle with speed $N^{1+\frac{\alpha}{p}}$ and good rate function $J_p$, where $\PP^N_{V,\beta}$ is defined in \eqref{defbeta}.  If $p$ is even,
$$
J_p(x) =  
\begin{cases}
b\left(x-\langle \sigma_{\beta}^V, x^p\rangle\right)^{\frac{\alpha}{p}}& \text{ if } x\geq \langle \sigma_{\beta}^V, x^p\rangle, \\
 +\infty &\text{ otherwise,}
\end{cases}$$
where $\langle\sigma_{\beta}^V, x^p\rangle$ denotes the $p^{\text{th}}$ moment of the equilibrium measure of  $\PP_{V, \beta}^N$, and if $p$ is odd, $J_p$ is defined by,
$$\forall x \in \RR, \ J_p(x) = b\left|x-\langle\sigma_{\beta}^V, x^p\rangle\right|^{\frac{\alpha}{p}}.$$
\end{The}

\begin{Rem} The rate function in Theorem \ref{ldp1} is the same as the rate function of the LDP of $$(\langle \sigma_{\beta}^V, x^p\rangle + \frac{1}{N}\sum_{i=1}^N X_i^p)_{N\in \NN},$$ where $(X_i)_{i\geq1}$ are i.i.d random variables with law $e^{-NV(x)}dx/Z_V$, where we denote $Z_V = \int e^{-NV(x)} dx$ (see Lemma \ref{ldp2}). This indicated that the logarithmic interaction between the particle of the Coulomb gas become negligible when one is considering large deviations of $m_{p,N}$.
\end{Rem}

\begin{Rem}
One can also derive a large deviations principle of the even moments of the empirical measure, say $m_{2p,N}$, under  $\langle \sigma_{\beta}^V, x^{2p}\rangle$, with speed $N^2$. Indeed, the proof of the large deviations of the empirical measure yields the asymptotics of the partition function $Z^N_{V,\beta}$ at the exponential scale $N^2$  (see \cite[Theorem 2.6.1]{Guionnet}). But the scaled logarithmic moment generating function of $m_{2p,N}$ at some $t<0$, is finite, and is actually equal to the partition function $Z^N_{V -tx^{2p}, \beta}$, associated with the potential $V-tx^{2p}$. Gärtner-Ellis theorem (see \cite[Theorem 2.3.6]{Zeitouni}), thus yields a large deviations principle with speed $N^2$ of $\langle L_N, x^{2p}\rangle$ on $(-\infty, \langle \sigma_{\beta}^V, x^{2p}\rangle )$. 
\end{Rem}

Let us introduce now the model of Wigner matrices. The Wigner matrices and the $\beta$-ensembles are linked through the GOE, GUE and GSE, which form a $\beta$-ensemble for a quadratic potential and $\beta=1,2,4$ respectively. More generally, let $(X_{i,j})_{i<j}$ be independent and identically distributed (i.i.d) complex-valued centered random  variables, and let $(X_{i,i})_{i\geq 1}$ be i.i.d real-valued centered random variables such that for any $k\in \NN$,
$$ \max\left( \EE|X_{1,1}|^k,\EE|X_{1,2}|^k \right) < +\infty.$$
Let $X(N)$ be the $N \times N$ Hermitian matrix with up-diagonal entries $(X_{i,j})_{1\leq i \leq j\leq N}$. We call such a sequence $(X(N))_{N\in \NN}$, a Wigner matrix. In the following, we will drop the $N$ and write $X$ instead of $X(N)$. 

Consider now the normalized random matrix $X_N = X/\sqrt{N}$. Let $\lambda_i$ denote the eigenvalues of $X_N$, with $\lambda_1 \leq \lambda_2 \leq ...\leq \lambda_N$. We define $L_{N}$ the empirical spectral measure of $X_N$ by,
$$L_{N} = \frac{1}{N}\sum_{i=1}^N \delta_{\lambda_i}.$$ 
Wigner's theorem (see \cite{Wigner}, \cite[Theorem 2.1.1, Exercice 2.1.16]{Guionnet}, \cite[Theorem 2.5]{Silverstein}) states that,
$$ L_{N}   \underset{N\to +\infty}{\rightsquigarrow} \sigma_{sc} \text{ a.s }$$
where $ \sigma_{sc}$ denotes the semicircular law, that is,
$$  \sigma_{sc}= \frac{1}{2\pi} \Car_{|x|\leq 2} \sqrt{4-x^2} dx,$$
and for any $p\in \NN$, almost surely, it holds
$$\langle L_{N},x^p\rangle = \frac{1}{N}\tr X_N^p  \underset{N\to+\infty}{\longrightarrow } 
\begin{cases}
C_{p/2}&\text{ if $p$ is even,}\\
0 & \text{ if $p$ is odd,}
\end{cases} 
$$
 and in the case $p$ is even, $C_{p/2}$ denotes the $\left(\frac{p}{2}\right)^{\text{th}}$ Catalan number, which is also the $p^{\text{th}}$ moment of the semicircular law.

In the following we will denote for any $A\in \M_N(\CC)$, the normalized trace $\tr_N A$. In the case of Gaussian Wigner matrices, we have the following result.
\begin{The}\label{LDPtrgauss}Let $p\in \NN$, $p\geq 3$.
Let $X$ be a Wigner matrix with Gaussian entries. We assume that $X_{1,1}$ is a centered real Gaussian variable of variance $\sigma^2$, and $X_{1,2}$ is a centered Gaussian variable, possibly complex, such that $\EE|X_{1,2}|^2=1$. The sequence $\left(\tr_N X_N^p \right)_{N\in \NN}$,
follows a LDP with speed $N^{1+\frac{2}{p}}$, and good rate function $J_p$.
If $p$ is even, $J_p$ is given by,
$$\forall x \in \RR, \ J_p(x) = \begin{cases}
\frac{1}{2}\min\left( \frac{1}{\sigma^2} , \frac{\beta}{2}\right)\left( x-C_{p/2}\right)^{\frac{2}{p}}& \text{ if } x\geq C_{p/2},\\
+\infty & \text{ otherwise,}
\end{cases}
$$
where $C_{p/2}$ denotes the $\left( \frac{p}{2}\right)^{\text{ th}}$ Catalan number, and if $p$ is odd, 
$$\forall x \in \RR, \ J_p(x) = \frac{1}{2}\min\left( \frac{1}{\sigma^2} ,\frac{\beta}{2}\right)
\left| x\right|^{\frac{2}{p}},$$
where $\beta =1$, if $X_{1,2}$ is a real Gaussian variable, and $\beta=2$ if $X_{1,2}$ is a complex Gaussian variable.
\end{The}

We consider now the so-called model of Wigner matrices without Gaussian tails investigated in \cite{Bordenave}. We recall in the following definition this model.
\begin{Def}\label{WG}
We say that $X$ is a Wigner matrix without Gaussian tail, if $X$ is a Wigner matrix such that there exist $\alpha \in (0,2)$ and $a, b \in (0, +\infty)$ such that,
\begin{equation} \label{queue de distrib}\lim_{t\to +\infty} -t^{-\alpha}\log \PP\left(|X_{1,1}|>t\right) = b,\end{equation}
$$ \lim_{t\to +\infty} -t^{-\alpha} \log \PP\left(|X_{1,2}|>t\right) = a.$$
Moreover, we assume that there are two probability measures on $\mathbb{S}^1$, $\upsilon_1$ and $\upsilon_2$,  and $t_0>0$, such that for all $t\geq t_0$ and any measurable subset $U$ of $\mathbb{S}^1$,
\begin{equation*} \label{decouplage module argument}\PP\left( X_{1,1}/|X_{1,1}| \in U, |X_{1,1}|\geq t\right) = \upsilon_1(U) \PP\left(|X_{1,1}|\geq t\right),\end{equation*}
$$\PP\left( X_{1,2}/|X_{1,2}| \in U, |X_{1,2}|\geq t\right) = \upsilon_2(U) \PP\left(|X_{1,2}|\geq t\right).$$
We denote the normalized matrix $X_N = X/\sqrt{N}$.
\end{Def}

With this definition, we can now state the following result.
\begin{The}\label{LDPtraceWG} Let $p\in \NN$, $p\geq3$. 
Let $X$ be a Wigner matrix without Gaussian tail. The sequence $(\tr_N X_N^p)_{N\geq 1}$ satisfies a large deviations principle with speed $N^{\alpha\left( \frac{1}{2} + \frac{1}{p}\right)}$ and good rate function $J_p$. If $p$ is even, $J_p$ is given by
$$\forall x\in \RR, \  J_p(x) = \begin{cases}
c_p\left(x-C_{p/2}\right)^{2/p}& \text{if } x\geq C_{p/2},\\
+\infty & \text{otherwise},
\end{cases}$$
where $C_{p/2}$ denotes the $\left( \frac{p}{2}\right)^{\text{th}}$ Catalan number, and if $p$ is odd, the rate function $J_p$ is given by
$$ \forall x \in \RR, \ J_p(x) =
c_p|x|^{2/p},$$
where  $c_p$ is a constant depending on $p$, $\alpha$, $a$ and $b$. 

Furthermore, if $\alpha \in (0,1]$ and $p$ is even, then $c_p = \min \left( b ,2^{-\alpha/p}a\right)$.
\end{The}

\begin{Rem}
Note that for $p=2$, the trace of $X^2$ is a sum of i.i.d random variables, so that one can apply Cramer's theorem (see \cite[Theorem 2.2.3]{Zeitouni}) in the case where the entries have finite Laplace transform, or Nagaev's truncation approach (see \cite{Nagaev} or \cite{Gantert}) in the case where the entries have a tail distribution behaving as $e^{-ct^{\alpha}}$, with some $c>0$, and $\alpha \in (0,2)$.
\end{Rem}
\begin{Rem}
The constant $c_p$ appearing in Theorem \ref{LDPtraceWG} is the solution of an optimization problem described in \eqref{constoptim}. We solve this optimization problem in section \ref{CompJ1}, in the easiest case when $\alpha \in (0,1]$ and $p$ is even, and we give a lower a bound and upper bound in the case $p$ is even and $\alpha \in (0,2)$.
\end{Rem}

\section*{Acknowledgements}
I would like to thank my supervisor Charles Bordenave for his inspiring advice, and all the fruitful conversations which helped me to build the present paper. Furthermore, I would like to thank Mireille Capitaine and Catherine Donati-Martin for introducing us to the problem of large deviations of traces of random matrices and providing us with key references and remarks, and also Michel Ledoux for pointing out precious references.

\section{The Gaussian case}\label{Gaussian}
We study in this section the question of the large deviations of the moments of the empirical spectral measure of a Wigner matrix with Gaussian entries. We will use an approach which is greatly inspired from Borell's proof of the LDP for Weiner chaos (see \cite{Borell} \cite{Borell2}), and especially Ledoux's exposition in \cite[Section 5, Theorem 5.1]{Ledouxflour}.

As we will see in the proof, the deviations of the trace are created by translations of $X$ of the form $N^{1/2+1/p}H$, where $H$ is with bounded Hilbert-Schmidt norm. One of the central argument relies on the following lemma.

\begin{Lem}\label{lemconv}Let $\beta \in \{1,2\}$. We denote by $\mathcal{H}_N^{(\beta)}$ the set of symmetric matrices of size $N$ when $\beta=1$, and Hermitian matrices when $\beta=2$.
Let $|| \ ||_2$ denote the Hilbert-Schmidt norm on $\mathcal{H}_N^{(\beta)}$. Let $X$ be a Wigner matrix whose entries are centered and have finite $(p+1)^{\text{th}}$ moment. For any $r>0$,

\begin{equation}\label{convprob}
\sup_{\underset{H\in \mathcal{H}_N^{(\beta)}}{||H||_2\leq r }} \left|  \tr_N\left( \frac{X}{\sqrt{N}} + N^{1/p} H \right)^p - \langle \sigma_{sc}, x^p\rangle -\tr H^p \right| \underset{N\to +\infty}{\longrightarrow} 0,\end{equation}
in probability, where $\langle \sigma_{sc}, x^p\rangle$ denotes the $p^{\text{th}}$ moment of the semicircular law.
\end{Lem}

\begin{proof}
By Wigner's theorem (see \cite[Lemmas 2.1.6, 2.1.7]{Guionnet}) and Jensen's inequality, we only have to prove that for any $Y,H \in \mathcal{H}^{(\beta)}_N$,
$$  | \tr(Y+H)^p - \tr Y^p-\tr H^p | \leq 2^p \max_{1\leq k\leq p-1}\Big\{ \Big( \tr|Y|^{p+1}\Big)^{\frac{k}{p+1}} (\tr H^2)^{\frac{p-k}{2}}\Big\}.$$ 
Let $Y,H \in \mathcal{H}^{(\beta)}_N$. Expanding the trace, and using the cyclicity of the trace, it suffices to prove that for any $s \in \{1,...,p\}$, $n_1,...,n_s\in \NN$, $m_1,...,m_s\in \NN$, such that $\sum_{i=1}^s n_i + \sum_{j=1}^s m_j = p$, we have
$$\Big|\tr\left( Y^{n_1} H^{m_1}...Y^{n_s} H^{m_s} \right) \Big|\leq \Big(\tr|Y|^{p+1}\Big)^{\frac{k}{p+1}} \Big( \tr |H|^2\Big)^{\frac{p-k}{2}},$$
with $k =\sum_{i=1}^s n_i$.
Applying Hölder's inequality (see \cite[Corollary IV.2.6]{Bhatia}) with the exponents $ \frac{p+1}{n_1}, \alpha, \frac{p+1}{n_2},..., \alpha, \frac{p+1}{n_s}$, with $\alpha$ such that \begin{equation} \label{expo} \frac{s}{\alpha}=  1-  \sum_{i=1}^s \frac{n_i}{p+1},\end{equation} we get,
$$\big|\tr\left( Y^{n_1} H^{m_1}...Y^{n_s} H^{m_s} \right)\big| \leq  \Big( \tr |Y|^{p+1}\Big)^{\frac{1}{p+1} \sum_{i=1}^s n_i} \prod_{j=1}^s \left( \tr \left|H\right|^{\alpha m_i} \right)^{\frac{1}{\alpha}}.$$
Note that  when $s\geq 2$, we have from \eqref{expo}, $\alpha \geq2$. If $s =1$ and $m_1=1$, then as $p\geq 3$, \eqref{expo} yields $\alpha m_1 \geq2$. In any cases, $\alpha m_i \geq 2$ for any $i\in \{1,...,s\}$. Therefore, for  all $i\in \{1,...,s\}$,
$$ \tr \left|H\right|^{\alpha m_i}  \leq \Big(\tr H^2\Big)^{\frac{\alpha m_i}{2}}.$$
Thus,
$$\big|\tr\left( Y^{n_1} H^{m_1}...Y^{n_s} H^{m_s} \right)\big| \leq  \Big( \tr |Y|^{p+1}\Big)^{\frac{1}{p+1} \sum_{i=1}^s n_i} \left( \tr H^2 \right)^{\frac{1}{2}\sum_{i=1}^s m_i},$$
which gives the claim.
\end{proof}

Before giving the proof of Theorem \ref{LDPtrgauss}, we will give back to the rate function defined in the statement of Theorem \ref{LDPtrgauss} its variational form, which is the following.

\begin{Lem}\label{computfntauxGauss}
Define  
$$\forall H\in \mathcal{H}_n^{(\beta)}, \ q(H) =\frac{1}{\sigma^2} \sum_{i=1}^{+\infty} H_{i,i}^2+ \beta \sum_{i<j} \left| H_{i,j} \right|^2, \ \phi(H) = \langle \sigma_{sc},x^p\rangle + \tr H^p.$$
Then for all $s\in \RR$,
$$J_p(s) = \inf \left\{ \frac{1}{2}q(H) : s =\phi(H), H\in \cup_{n\geq1} \mathcal{H}_n^{(\beta)} \right\},$$
where $J_p$ is the rate function defined in the statement of Theorem \ref{LDPtrgauss}.

\end{Lem}

\begin{proof}
For any $H \in  \mathcal{H}_n^{(\beta)}$, we have
$$ q(H) \geq \min\left( \frac{1}{\sigma^2}, \frac{\beta}{2} \right) \tr H^2.$$
As $p\geq 2$, we get
\begin{equation} \label{borne} q(H) \geq \min\left( \frac{1}{\sigma^2}, \frac{\beta}{2} \right) \left|\tr H^p \right|^{2/p}. \end{equation}
This yields for any $s\in \RR$,
$$ \inf \left\{\frac{1}{2} q(H) : s =\phi(H), H\in \cup_{n\geq1} \mathcal{H}_n^{(\beta)} \right\} \geq J_p(s).$$
Note, that for any $s\in \RR$, 
$$q(s-\langle\sigma_{sc}, x^p\rangle)=\frac{1}{\sigma^2}|s-\langle \sigma_{sc},x^p\rangle|^{2/p}.$$ 
Assume $p$ is even. Let $s\in [\langle\sigma_{sc},x^p\rangle, +\infty)$, and $n\in \NN$. Define
\begin{equation}\label{defH}
H
= \left(
     \raisebox{0.5\depth}{%
       \xymatrixcolsep{1ex}%
       \xymatrixrowsep{1ex}%
       \xymatrix{
        0 \ar @{.}[ddddrrrr]& \lambda \ar @{.}[rrr] \ar @{.}[dddrrr] &  & &\lambda  \ar @{.}[ddd]  \\
         \lambda \ar@{.}[ddd] \ar@{.}[dddrrr]& & & & \\
         &&&& \\
         &&&&\lambda \\
         \lambda \ar@{.}[rrr] & & &\lambda& 0
       }%
     }
   \right)\in  \mathcal{H}_n^{(\beta)}, \end{equation}
with $\lambda = \left( \frac{s- \langle\sigma_{sc},x^p\rangle}{(n-1)^p+(n-1)} \right)^{1/p}$. We have $ s = \langle \sigma_{sc}, x^p\rangle +\tr H^p$, and
 $$ q(H) = \beta \frac{n(n-1)}{2} \left( \frac{s- \langle\sigma_{sc},x^p\rangle}{(n-1)^p+(n-1)} \right)^{2/p}\underset{n\to +\infty}{\longrightarrow} \frac{\beta}{2}\left( s-\langle\sigma_{sc},x^p\rangle \right)^{2/p}.$$
This yields for any $s\in [\langle\sigma_{sc},x^p\rangle, +\infty)$,
$$ \inf \left\{ q(H) : s =\phi(H), H\in \cup_{n\geq1} \mathcal{H}_n^{(\beta)} \right\}=J_p(s).$$
For $s  < \langle \sigma_{sc}, x^p \rangle $, the above inequality is true, since both of the quantities are equal to $+\infty$. Assume now $p$ is odd. Let $s\in\RR$, and define $H\in \mathcal{H}^{(\beta)}_n$ as in \eqref{defH} but with $\lambda =\sg(s) \left( \frac{|s|}{(n-1)^p-(n-1)} \right)^{1/p}$, so that $ s = \tr H^p$. We have
$$ q(H) = \beta \frac{n(n-1)}{2} \left( \frac{|s|}{(n-1)^p-(n-1)} \right)^{2/p}\underset{n\to +\infty}{\longrightarrow}\frac{\beta}{2}|s|^{2/p}.$$
As in the case where $p$ is even, this yields for any $s\in\RR$,
$$ \inf \left\{ \frac{1}{2}q(H) : s =\phi(H), H\in \cup_{n\geq1} \mathcal{H}_n^{(\beta)} \right\} = J_p(s),$$
which ends the proof.
\end{proof}

We are now ready to give a proof of Theorem \ref{LDPtrgauss}. As in Borell's proof of the LDP of Weiner chaoses (see for example \cite[Theorem 5.1]{Ledouxflour}), the proof of the upper bound relies on a reformulation of the deviations of the trace in terms of an enlargement of a properly chosen event. Then, the Gaussian isoperimetric inequality allows us to estimate the probability of such enlargement. Similarly as in Borell's proof of the lower bound, we use here a kind of finite-dimensional version of Cameron-Martin formula.
\begin{proof}[Proof of Theorem \ref{LDPtrgauss}]We closely follow the outline of proof of the large deviations of Weiner chaoses in \cite[Section 5, Theorem 5.1]{Ledouxflour}.

\textbf{Upper bound}
Let $A$ be a closed subset of $\RR$.  We can assume without loss of generality that $\inf_A J_p>0$, otherwise there is nothing to prove.

 Let $0< r < \inf_A J_p$.
Using the notation of Lemma \ref{computfntauxGauss}, we define for any $N\in \NN$, 
$$\mathcal{K}_N = \left\{ H \in  \mathcal{H}_N^{(\beta)} : q(H) \leq 1 \right\}, \text{ and }   \mathcal{K} = \left\{ H \in  \cup_{n\geq 1}\mathcal{H}_n^{(\beta)} : q(H) \leq 1 \right\}.$$
We claim that,
$$ \overline{\phi\left(\sqrt{2r}\mathcal{K}\right)}\cap A = \emptyset.$$
Indeed,
if $s\in \overline{\phi\left(\sqrt{2r}\mathcal{K}\right)}$, we can find a sequence $(H_k)_{k\in \NN}$ in $\mathcal{K}$, such that 
$$ s = \lim_{k\to +\infty} \phi\left(\sqrt{2r} H_k \right).$$
As $J_p$ is lower semi continuous, we have
$$ J_p(s) \leq \liminf_{k\to +\infty} J_p\left(  \phi\left(\sqrt{2r} H_k \right) \right).$$
As $H_k \in \mathcal{K}$, we have $q(\sqrt{2r}H_k) \leq 2r$. From Lemma \ref{computfntauxGauss}, we get
$$J_p(s) \leq r.$$
This yields $s\notin A$.

From \eqref{borne}, we deduce that $ \overline{\phi\left(\sqrt{2r}\mathcal{K}\right)}$ is bounded. Thus it is a compact subset, which yields that there is some $\eta>0$ such that 
  $$ \left(\phi\left(\sqrt{2r}\mathcal{K}\right)+ B(0,\eta) \right)\cap A = \emptyset.$$
As $\mathcal{K}_N \subset \mathcal{K}$, we have for any $N \in \NN$,
  \begin{equation*}  \left(\phi\left(\sqrt{2r}\mathcal{K}_N\right)+ B(0,\eta) \right)\cap A = \emptyset.\end{equation*}
Observe here that $\eta$ does not depend on $N$. 
We deduce that
$$\PP\Big( \tr_N\Big( \frac{X}{\sqrt{N}}\Big)^p \in A \Big) \leq \PP\Big( \tr_N\left( \frac{X}{\sqrt{N}}\right)^p \notin \phi\Big(\sqrt{2r}\mathcal{K}_N\Big) + B(0,\eta) \Big).$$
Let 
$$V = \Big\{ Y\in \mathcal{H}_N^{(\beta)}: \sup_{H \in \mathcal{K}_N} \Big| \tr_N\left( \frac{Y}{\sqrt{N}} + N^{1/p}H\right)^p -  \langle\sigma_{sc},x^p\rangle  - \tr H^p \Big| < \eta \Big\}.$$
Then,
$$\PP\left( \tr_N\left( \frac{X}{\sqrt{N}}\right)^p \notin \phi\left(\sqrt{2r}\mathcal{K}_N \right) + B(0,\eta) \right) \leq \PP\left( X \notin V +\sqrt{2r} N^{1/2+1/p} \mathcal{K}_N\right).$$

By Lemma \ref{lemconv}, we know that for $N$ large enough, $\PP\left( X \in V \right) \geq 1/2$.  The Gaussian isoperimetric inequality (see \cite[Theorem 4.3]{Ledouxflour}) yields
$$\PP\left( X \notin V +\sqrt{2r} N^{1/2+1/p} \mathcal{K}_N\right) \leq\frac{1}{2} e^{-r N^{1+2/p}}.$$ 
Therefore,
$$\PP\left( \tr_N\left( \frac{X}{\sqrt{N}}\right)^p \in A \right) \leq  \frac{1}{2}e^{-r N^{1+2/p}}.$$ 
Thus, 
$$ \limsup_{N\to +\infty} \frac{1}{N^{1+2/p}} \log \PP\left(  \tr_N\left( \frac{X}{\sqrt{N}}\right)^p \in A \right) \leq -r.$$
Since the previous inequality is valid for any $0<r<\inf_A J_p$, this yields the upper bound of the LDP.

\textbf{Lower bound} Let $A$ be an open subset of $\RR$. Let $s \in A$. There is some $\eta>0$ such that $B(s,\eta) \subset A$. We can assume without loss of generality that $J_p(s) <+\infty$. Define for any $N\in \NN$,
$$\forall t \in \RR, \ J_{p,N}(t) = \inf_{H \in \mathcal{H}_N^{(\beta)} }  \left\{ \frac{1}{2}q(H) : t =\langle \sigma_{sc},x^p\rangle + \tr H^p  \right\}.$$
Let $\delta>0$. Due to Lemma \ref{computfntauxGauss}, for $N$ large enough, we have
$$  J_{p,N}(s) \leq  J_{p}(s) +\delta.$$ 
Let $r>0$ such that $\frac{r^2}{2}-\delta >J_p(s)$. We define the event
$$ V_r = \Big\{Y\in \mathcal{H}_N^{(\beta)} : \sup_{K\in r\mathcal{K}_N} \Big| \tr_N\Big( \frac{Y}{\sqrt{N}} +N^{1/p}K \Big)^p - \langle \sigma_{sc},x^p\rangle  - \tr K^p\Big| <\eta \Big\}.$$
Note that 
$$\mathcal{K}_N \subset \left( \min\Big( \frac{1}{\sigma^2}, \frac{\beta}{2} \Big) \right)^{-1/2} B_2,$$  
where $B_2$ denotes the unit ball of $\mathcal{H}_N^{(\beta)}$ for the Hilbert-Schmidt norm. Therefore Lemma \ref{lemconv} yields that for $N$ large enough, $\PP\left( X \in V_r \right) \geq 1/2$. 

As for $N$ large enough $J_{p,N} (s) \leq \frac{r^2}{2}$, we can write,
$$ J_{p,N}(s) = \inf_{H\in r\mathcal{K}_N}\Big\{ \frac{1}{2} q(H) : s = \langle \sigma_{sc},x^p\rangle + \tr H^p \Big\}.$$
 Let $H \in r\mathcal{K}_N$ be such that $s =  \langle \sigma_{sc}, x^p\rangle + \tr H^p$. 
Then,
\begin{equation} \label{lowerbound} \PP\left(    \tr_N\left( \frac{X}{\sqrt{N}}\right)^p \in A \right) \geq \PP\left(    \tr_N\left( \frac{X}{\sqrt{N}}\right)^p \in B(s,\eta) \right) = \PP\left( X \in V\right),\end{equation}
where
$$ V = \Big\{Y\in \mathcal{H}_N^{(\beta)} :  \Big| \tr_N\left( \frac{Y}{\sqrt{N}} \right)^p - \langle \sigma_{sc},x^p\rangle - \tr H^p\Big| <\eta \Big\}.$$
But,
\begin{align*}
\PP\left( X\in V \right) & = \PP\left( X-N^{\frac{1}{2}+\frac{1}{p}}H \in V- N^{\frac{1}{2}+\frac{1}{p}} H \right)  \\
& = \frac{1}{Z_N^{(\beta)}}\int_{V-N^{\frac{1}{2}+\frac{1}{p}}H} e^{-\frac{1}{2} q \left(Y+N^{\frac{1}{2}+\frac{1}{p}}H\right)} d\ell_N^{(\beta)}(Y),
\end{align*}
where $d\ell_N^{(\beta)}$ denotes the Lebesgue measure on $\mathcal{H}_N^{(\beta)} $, and $Z_N^{(\beta)} = \int e^{-\frac{1}{2}q(Y)} d\ell^{(\beta)}_N(Y)$.
We re-write this probability as,
$$\PP\left( X\in V \right) = e^{-\frac{1}{2}q(H) N^{1+\frac{2}{p}}} \EE\left( \Car_{\{X\in V-N^{\frac{1}{2}+\frac{1}{p}}H\}} e^{-N^{\frac{1}{2}+\frac{1}{p}} \Re \psi(H,Y)} \right),$$
where $\psi$ is the bilinear (or sesquilinear form if $\beta=2$) form associated to the quadratic form $q$.
Using Jensen's inequality, we get
\begin{align*}
\PP\left( X\in V \right)  & \geq e^{-\frac{1}{2}q(H) N^{1+\frac{2}{p}}} \PP\left(X\in V-N^{\frac{1}{2}+\frac{1}{p}}H\right)\\
& \times\exp\left( -N^{\frac{1}{2}+\frac{1}{p}}\EE\left(  \Re\psi(H,Y) \frac{ \Car_{\{X\in V-N^{\frac{1}{2}+\frac{1}{p}}H\}} }{\PP\left(X\in V-N^{\frac{1}{2}+\frac{1}{p}}H\right)} \right)\right).
\end{align*}
Using twice Cauchy-Schwarz inequality yields,
\begin{align*}\EE\left(- \Re\psi(H,X)\frac{ \Car_{\{X\in V-N^{\frac{1}{2}+\frac{1}{p}}H\}} }{\PP\left(X\in V-N^{\frac{1}{2}+\frac{1}{p}}H\right)} \right)&\geq - \frac{1}{ \PP\left(X\in V-N^{\frac{1}{2}+\frac{1}{p}}H\right)}\left(\EE\left(  \Re\psi\left(X,H\right)\right)^{2} \right)^{1/2}.\\
&= - \frac{1}{ \PP\left(X\in V-N^{\frac{1}{2}+\frac{1}{p}}H\right)} q(H)^{1/2}\left( \EE q(X)^2 \right)^{1/2}.\\
& =  - \frac{1}{ \PP\left(X\in V-N^{\frac{1}{2}+\frac{1}{p}}H\right)} q(H)^{1/2}.
\end{align*}
But $\PP\left(X\in V-N^{\frac{1}{2}+\frac{1}{p}}H\right) \geq  \PP\left( X\in V_r \right)\geq 1/2 $. Thus, we have
$$\PP\left( X\in V \right)  \geq \frac{1}{2} \exp\left(-\frac{1}{2}q(H) N^{1+\frac{2}{p}} -2q(H)^{1/2}N^{\frac{1}{2}+\frac{1}{p}} \right).$$
Since $H\in r\mathcal{K}_N$, we get
$$\PP\left( X\in V \right)  \geq \frac{1}{2} \exp\left(-\frac{1}{2}q(H) N^{1+\frac{2}{p}} -2rN^{\frac{1}{2}+\frac{1}{p}} \right).$$
As the above inequality is true for any $H\in r\mathcal{K}_N$ such that $s = \langle \sigma_{sc},x^p\rangle + \tr H^p$, we have
$$\PP\left( X\in V \right)  \geq \frac{1}{2} \exp\left(-J_{p,N}(s) N^{1+\frac{2}{p}} -2rN^{\frac{1}{2}+\frac{1}{p}} \right)\geq  \exp\left(-\left(J_{p}(s)+\delta\right) N^{1+\frac{2}{p}} -2rN^{\frac{1}{2}+\frac{1}{p}} \right).$$

We deduce from \eqref{lowerbound} that
$$\liminf_{N\to +\infty}  \frac{1}{N^{1+2/p}} \log \PP\left(  \tr_N\left( \frac{X}{\sqrt{N}} \right)^p \in A \right) \geq - J_p(s)-\delta.$$
Letting $\delta$ go to $0$, we get the lower bound.
\end{proof}

\newpage

\section{Large deviations of moments of the empirical measure of $\beta$-ensembles} \label{betaens}

We now give a proof of Theorem \ref{ldp1}. In order to ease the notation, we will write $\PP^N_{V}$ for $\PP^N_{V, \beta}$, as well as $Z^N_V$ instead of $Z^N_{V,\beta}$.

 \subsection{Deviations inequalities and convergence of the moments}
The first step of the proof of Theorem \ref{ldp1} will be to show, under the mild assumption \eqref{assumpgene}, the convergence in expectation, of the moments of the empirical measure towards the moments of the equilibrium measure $\sigma_{\beta}^V$. To do so, we will need a control on the tail probability of 
$$ \max_{1\leq i \leq N} |\lambda_i|,$$
under $\PP^N_{V}$. To this end we prove a more general deviations inequality, which will be crucial later.

\begin{Pro}\label{probatrou}Let $N\in \NN$, $N\geq 2$. Under assumption \eqref{assumpgene},
there is a constant $M_0>0$, depending only on $V$ and $\beta$, such that for any $M\geq M_0$  and $1\leq k \leq N$,
$$\PP^N_{V} \left( L_N\left( I_M^c\right) \geq \frac{k}{N}\right) \leq \exp\left( -C kNV_M\right),$$
where $I_M = [-M,M]$, $C$ is a positive constant depending on $V$ and $\beta$, and where $V_M = \inf_{|\lambda| \geq M} V(\lambda)$.
\end{Pro}

In order to prove this deviation inequality, we will need a rough control on the ratio of the partition functions $Z^N_{V}$ and $Z^{N-k}_{\frac{NV}{N-k}}$. This is the object of the following lemma.
\begin{Lem}\label{encadrefuncpart}
There are some constants $c_1, c_2$ depending on $V$ and $\beta$, such that for any $N \in \NN$, and $k\leq N$,
$$  c_1 Nk\leq \log \frac{Z^N_{V}}{Z^{N-k}_{\frac{NV}{N-k}}} \leq c_2N k,$$
where $Z^N_{V}$, and $Z^{N-k}_{\frac{NV}{N-k}}$ are defined in \eqref{partition func}.
\end{Lem}

\begin{proof}
From the invariance under permutation of the coordinates of the measures $\PP^{N}_{V}$ we have
$$\frac{Z^N_{V}}{Z^{N-k}_{\frac{NV}{N-k}}} =\frac{N !}{Z^{N-k}_{\frac{NV}{N-k}}}  \int_{|\lambda_1|\geq ...\geq |\lambda_{N}|} e^{- N\sum_{i=1}^{N} V\left( \lambda_{i}\right)} \prod_{1\leq i< j\leq N} \left| \lambda_{i}-\lambda_{j} \right|^{\beta} \prod_{i=1}^N d\lambda_i.
$$
Splitting the $\lambda_i$'s between the $k$ first largest in absolute value and the rest, and using again the invariance under permutation of the coordinates, we can bring out the measure $\PP^{N-k}_{ \frac{NV}{N-k}}$, which gives
\begin{align*}
 \frac{Z^N_{V}}{Z^{N-k}_{\frac{NV}{N-k}}} &=\frac{N !}{\left( N-k\right) !} \EE^{N-k}_{\frac{NV}{N-k}} \Big( \int_{|\lambda_1|\geq ...\geq |\lambda_{k}|} e^{- N\sum_{i=1}^{k} V\left( \lambda_{i}\right)} \prod_{1\leq i< j\leq k} \left| \lambda_{i}-\lambda_{j} \right|^{\beta} \\
 &\times e^{\beta\left(N-k\right) \sum_{i=1}^{k} \left\langle L_{N-k} , \log \left|\lambda_{i}-.\right| \right\rangle } \Car_{\mathrm{supp}\left( L_{N-k} \right) \subset \left[-\lambda_{k}, \lambda_{k}\right] }\prod_{i=1}^k d\lambda_i \Big ),
\end{align*}
where $L_{N-k} = \frac{1}{N-k} \sum_{i=k+1}^{N} \delta_{\lambda_{i}}$. We re-write this equality as the following,
 \begin{align*}
 \frac{Z^N_{V}}{Z^{N-k}_{\frac{NV}{N-k}}} & = \frac{N !}{\left( N-k\right) !}\EE^{N-k}_{\frac{NV}{N-k}} \Big( \int_{|\lambda_1|\geq ...\geq |\lambda_{k}|} e^{- k^2 \int_{x\neq y} f(x,y) dL_k(x) dL_k(y)} \\
 &\times e^{-\left(N-k\right) \sum_{i=1}^{k}\left( V\left(\lambda_i \right) -  \beta \left\langle L_{N-k} , \log \left|\lambda_{i}-.\right| \right\rangle\right) }\Car_{\mathrm{supp}\left( L_{N-k} \right) \subset \left[-\lambda_{k}, \lambda_{k}\right] } \prod_{i=1}^k e^{- V\left(\lambda_i \right) } d\lambda_i\Big ),
\end{align*}
with $L_k = \frac{1}{k}\sum_{i=1}^k \delta_{\lambda_i}$, and $f(x,y) = \frac{1}{2}V(x)+ \frac{1}{2}V(y) - \frac{\beta}{2} \log \left|x-y\right|$. Note that from the assumption \eqref{assumpgene} on $V$, we have
$$ c :=\inf \{ f(x,y) : x \neq y \} > -\infty, \ c' := \inf\left\{ V(x) - \beta \log \left|x-y\right| : |y|\leq |x| \right\} > -\infty.$$ 
Thus,
$$
 \frac{Z^N_{V}}{Z^{N-k}_{\frac{NV}{N-k}}}  \leq \binom{N}{k}e^{-k^2 c} e^{-(N-k)kc'} \left(\int e^{-V(x)}dx\right)^k.
$$
As $\binom{N}{k} \leq N^k$, we get 
$$ \frac{Z^N_{V}}{Z^{N-k}_{\frac{NV}{N-k}}}  \leq e^{c_2 Nk},$$ with $c_2$ some constant depending on $V$ and $\beta$.

For the lower bound, we write similarly as for the upper bound,
\begin{align*}
\log  \frac{Z^N_{V}}{Z^{N-k}_{\frac{NV}{N-k}}}  & = \log \EE^{N-k}_{\frac{NV}{N-k}} \Big( \int e^{-(N-1)  \sum_{i=1}^k V(\lambda_i)  } \prod_{1\leq i<j\leq k} \left|\lambda_i - \lambda_j \right|^{\beta}\nonumber  \\
 &\times e^{\beta\left(N-k\right) \sum_{i=1}^{k} \left\langle L_{N-k} , \log \left|\lambda_{i}-.\right| \right\rangle } \prod_{i=1}^k e^{- V\left(\lambda_i \right) } d\lambda_i\Big).
\end{align*}
 Using twice Jensen's inequality, we get
\begin{align}
\log  \frac{Z^N_{V}}{Z^{N-k}_{\frac{NV}{N-k}}}  
& \geq  \EE^{N-k}_{\frac{NV}{N-k}} \Big( \log \int e^{-(N-1)  \sum_{i=1}^k V(\lambda_i)  } \prod_{1\leq i<j\leq k} \left|\lambda_i - \lambda_j \right|^{\beta} \nonumber  \\
 &\times e^{\beta\left(N-k\right) \sum_{i=1}^{k} \left\langle L_{N-k} , \log \left|\lambda_{i}-.\right| \right\rangle } \prod_{i=1}^k \frac{e^{- V\left(\lambda_i \right) }}{\int e^{-V(x) } dx} d\lambda_i\Big ) + k\log \left( \int e^{-V(x) }dx \right).\nonumber \\
&\geq -(N-1)k \Big( \int V(\lambda) \frac{e^{-V(\lambda) } d\lambda}{\int e^{-V(x)} dx }\Big) + 
\frac{k(k-1)\beta}{2} \Big( \int \log \left|\lambda- \mu \right| \frac{e^{-V(\lambda) -V(\mu)}d\lambda d\mu}{\left(\int e^{-V(x)}dx \right)^2}\Big)\nonumber \\
& +\beta k (N-k)  \EE^{N-k}_{\frac{NV}{N-k}}  \Big( \int \left\langle L_{N-k} , \log \left|\lambda-.\right| \right\rangle  \frac{e^{-V(\lambda) }d\lambda}{ \int e^{-V(x)} dx }\Big)+ k\log \Big( \int e^{-V(x) }dx \Big).\nonumber
\end{align}
But for any $\mu \in \RR$,
\begin{align*}
\int \log |\lambda -\mu| e^{-V(\lambda) } d\lambda &= \int_0^{+\infty} \log x \left( e^{-V(\mu +x) } + e^{-V(\mu-x)} \right) dx \\
&\geq \int_0^1 \log x \left( e^{-V(\mu +x) } + e^{-V(\mu-x)} \right) dx.
\end{align*}
As $\inf V < -\infty$, we have
$$ \int \log |\lambda -\mu| e^{-V(\lambda) } d\lambda\geq 2e^{-\inf V} \int_0^1 \log(x) dx =  -2e^{-\inf V}.$$
Thus,
$$\EE^{N-k}_{\frac{NV}{N-k}}  \left( \int \left\langle L_{N-k} , \log \left|\lambda-.\right| \right\rangle  \frac{e^{-V(\lambda)} d\lambda}{ \int e^{-V(x)} dx }\right) \geq -\frac{2e^{-\inf V}}{\int e^{-V(x) } dx }.$$ 
We can conclude that 
$$\log  \frac{Z^N_{V}}{Z^{N-k}_{\frac{NV}{N-k}}} \geq c_1Nk,$$
with $c_1$ a constant depending on $V$ and $\beta$.

\end{proof}

We are now ready to give a proof of Proposition  \ref{probatrou}.

\begin{proof}[Proof of Proposition \ref{probatrou}]
We can write as in the proof of Lemma \ref{encadrefuncpart},
\begin{align*}
\PP^N_{V} &\left( L_N\left( I_M^c\right) \geq  \frac{k}{N}\right)  \leq \frac{N !}{\left( N-k\right) !}  \frac{ Z^{N-k}_{\frac{NV}{N-k}}} {Z^N_{V}}    \\
 & \times \EE^{N-k}_{\frac{NV}{N-k}} \Big( \int_{|\lambda_1|\geq ...\geq |\lambda_{k}|\geq M}  e^{- N \sum_{i=1}^{k}V\left(\lambda_{i}\right)} \prod_{1\leq i< j\leq k} \left| \lambda_{i } -\lambda_{j}\right|^{\beta}\\
& \times  e^{\beta \left( N-k \right) \sum_{i=1}^{k}\left\langle L_{N-k} , \log \left|\lambda_{i}-.\right| \right\rangle } \Car_{\mathrm{supp}\left( L_{N-k} \right) \subset \left[-\lambda_{k}, \lambda_{k}\right]} \prod_{i=1}^k  d\lambda_i\Big),
 \end{align*}
with $L_{N-k} = \frac{1}{N-k} \sum_{i=k+1}^{N}\delta_{\lambda_i}$ . 

As for all $x, y\in \RR$, $\log |x-y| \leq \log \left( 1+|x|\right) + \log \left(1+|y|\right)$, and for any $|x| \leq |y|$, $\log |x-y| \leq \log 2 + \log(1+|x|)$,  we get 
\begin{align*}
\PP^N_{V} \left( L_N\left( I_M^c\right) \geq  \frac{k}{N}\right) & \leq \frac{N !}{\left( N-k\right) !}\frac{ Z^{N-k}_{\frac{NV}{N-k}}} {Z^N_{V}}  e^{k(N-k)\log 2}   \\
& \times \EE^{N-k}_{\frac{NV}{N-k}} \Big( \int_{|\lambda_1|\geq ...\geq |\lambda_{k}|\geq M}  e^{- N \sum_{i=1}^{k}V\left(\lambda_i\right)} e^{  \beta k \sum_{i=1}^{k} \log \left( 1+ \left| \lambda_{i}\right| \right)} \\
& \times e^{\beta \left( N- k\right) \sum_{i=1}^{k} \log \left( 1+\left|\lambda_{i} \right| \right) } \Car_{\mathrm{supp}\left( L_{N-k} \right) \subset \left[-\lambda_{k}, \lambda_{k}\right]} \prod_{i=1}^k  d\lambda_i\Big).
\end{align*}
From  \eqref{assumpgene}, we deduce that there is some $c_{0}>0$, such that for $|y|$ large enough,
$$  V(y) - \beta \log \left( 1+ |y| \right) \geq c_0 V(y).$$
Thus, for $M$ large enough,
\begin{align*}
\PP^N_{V} \left( L_N\left( I_M^c\right) \geq  \frac{k}{N}\right)
& \leq \frac{N !}{\left( N-k\right) !} \frac{ Z^{N-k}_{\frac{NV}{N-k}}} {Z^N_{V}}    \int_{|\lambda_1|\geq ...\geq |\lambda_{k}|\geq M}  e^{-c_0 N \sum_{i=1}^{k}V\left(\lambda_i\right)}  \prod_{i=1}^k  d\lambda_i\\
& =   \binom{N}{k}\frac{ Z^{N-k}_{\frac{NV}{N-k}}} {Z^N_{V}}    \left(\int_{\left|\lambda\right|\geq M}  e^{-c_0N V\left(\lambda\right)}   d\lambda \right)^k.
\end{align*}
But,
$$ \int_{\left|\lambda\right|\geq M}  e^{-c_0N V\left(\lambda\right)}   d\lambda \leq
e^{-c_0(N-1) V_M } \int e^{-V(\lambda) } d\lambda \leq  c_3e^{-\frac{c_0}{2}N V_M},$$
with some constant $c_3>0$, and where we used in the last inequality the fact that $N\geq 2$.
We deduce from Proposition \ref{encadrefuncpart} that for $M$ large enough,
$$
\PP^N_{V} \left( L_N\left( I_M^c\right) \geq  \frac{k}{N}\right)  \leq   (c_3N)^{k} e^{kNc_2}e^{-\frac{c_0}{2}kN V_M }.$$
As $ \lim_{M\to +\infty} V_M = +\infty$, we can find some constants $M_0>0$, and $C>0$, depending on $V$ and $\beta$, such that for any $M>M_0$, 
$$
\PP^N_{V} \left( L_N\left( I_M^c\right) \geq  \frac{k}{N}\right)  \leq  e^{-CkNV_M}.$$

%
\end{proof}

As a consequence of the previous Proposition \ref{probatrou}, we have the convergence of the expectation under $\PP^N_{V}$, of the moments of the empirical measure, as stated in the next corollary.
\begin{Cor}\label{convmoment}
Under assumption \eqref{assumpgene}, we have for any $p \in \NN$, 

$$ \EE^N_{V}\left\langle L_N, x^p\right \rangle \underset{N\to +\infty}{\longrightarrow} \langle \sigma_{\beta}^V, x^p\rangle,$$
where $\EE_{V}^N$ denotes the expectation with respect to $ \PP^N_{V}$.
\end{Cor}

\begin{proof}
Since $(L_N)_{N\geq 1}$ follows  a LDP with speed $N^2$ (see \cite[Theorem 2.6.1]{Guionnet}), and rate function whose minimum is achieved at $\sigma_ {\beta}^V$, we deduce that $(L_N)_{N\in \NN}$ converges weakly in probability to $\sigma_{\beta}^V$ under $\PP^N_{V}$. Thus, it is enough to show that for any $k \in \NN$, 
$$ \sup_{N\geq N_0}\EE^N_{V}\langle L_N, \left|x\right|^k\rangle < +\infty,$$
for some  $N_0\geq 1$.

Let $k \in \NN$. We have $\langle L_N,|x|^k\rangle \leq \max_{1\leq i\leq N} |\lambda_i|^k$. Besides, we know by Proposition \ref{probatrou} that
$$ \PP^N_{V} \left( \max_{1\leq i\leq N} |\lambda_i| > M \right) \leq e^{-CNV_M},$$
for any $M>M_0$, where $C$ and $M_0$ are some positive constants. Thus,
$$ \EE^N_{V}  \max_{1\leq i\leq N} |\lambda_i|^k \leq M_0^k + \int_{M_0}^{+\infty} k x^{k-1} e^{-CNV_x}dx.$$
By assumption we know that for $|x|$ large enough, $ V_x \geq \beta' \log |x|$, with  $\beta'>1$, so that for $M_0$ large enough,
$$ \EE^N_{V}  \max_{1\leq i\leq N} |\lambda_i|^k \leq M_0^k +  \int_{M_0}^{+\infty} k x^{k-1}x^{-C\beta' N} dx.$$
We deduce that for $N\geq (k+1)/C\beta'$, and $M_0$ large enough,
\begin{equation} \label{momentvpmax} \EE^N_{V}  \max_{1\leq i\leq N} |\lambda_i|^k \leq M_0^k +  \int_{M_0}^{+\infty} k x^{-2} dx=M_0^k+\frac{k}{M_0},\end{equation}
which yields the claim.
\end{proof}

\subsection{An exponential equivalence}

The goal of this section is to prove that the large deviations of $m_{p,N}$ are due to the deviations of the $\log N$ largest in absolute value $\lambda_i$'s . More precisely, we will prove the following proposition.

\begin{Pro}\label{expoequiv}For any $p\in \NN$, $p>\alpha$, and $\lambda_1,...,\lambda_N \in \RR$, we denote by $T_{p,N} $ the truncated moment
$$T_{p,N} = \frac{1}{N}\sum_{i=1}^{\log N} {\lambda_{i}^*}^p,$$
where $\lambda_1^*,...,\lambda_N^*$ is the rearrangement of the $\lambda_i$'s by decreasing absolute values.
 Under the notation and assumption of Theorem \ref{ldp1}, we have for any $t>0$,
$$ \lim_{N \to +\infty} \frac{1}{N^{1+\alpha/p}} \log \PP^N_{V} \left( \left| m_{p,N} -\langle \sigma_{\beta}^V,x^p\rangle - T_{p,N} \right|> t \right) = - \infty.$$

\end{Pro}

As a consequence of Proposition \ref{probatrou} and Corollary \ref{convmoment}, we have the following result.

\begin{Pro}\label{convmomenttronq}
Under assumption \eqref{assumpgene}, we have
$$ \frac{1}{N} \EE^N_{V} \Big( \sum_{i=\log N +1}^{N} {\lambda_i^*}^p\Big) \underset{N\to +\infty}{\longrightarrow}\langle \sigma_{\beta}^V,x^p\rangle.$$
\end{Pro}

\begin{proof}
Due to Corollary \ref{convmoment}, we only need to prove 
$$ \frac{1}{N} \sum_{i=1}^{\log N}  \EE^N_{V}{\lambda_i^*}^p \underset{N\to +\infty}{\longrightarrow} 0.$$
From \eqref{momentvpmax} we have
\begin{equation} \label{momentborne} \sup_{N\geq N_0} \EE^N_{V} | \lambda_1^*|^p < +\infty,\end{equation}
with $N_0 \in \NN$. Thus for any $N\geq N_0$,
$$ \Big|\frac{1}{N} \sum_{i=1}^{\log N}  \EE^N_{V}{\lambda_i^*}^p \Big| \leq \frac{\log N}{N}  \sup_{N\geq N_0} \EE^N_{V}| \lambda_1^*|^p \underset{N\to +\infty}{\longrightarrow} 0.$$

\end{proof}
Due to the previous proposition, in order to prove  Proposition \ref{expoequiv}, it suffices to show that 
$$ \frac{1}{N}\sum_{i=\log N +1}^{N} {\lambda_i^*}^p $$
concentrates at a speed higher than $e^{-N^{1+\alpha/p}}$. To this end, we will use concentration inequalities for $\alpha$-convex measures from \cite{Bobkov}. This is the object of the following proposition.
\begin{Pro}\label{conckconv} 
Let $\alpha \geq 2$. Let $g : \RR^N \to \RR$ be a $1$-Lipschitz function with respect to $||\ ||_{\alpha}$.  Under the notation and assumption of Theorem \ref{ldp1}, we have for every $t>0$,
$$\PP_{V}^N\left( g- \EE_{V}^N g > t\right) \leq  \exp\left( -\frac{bNt^{\alpha} }{2^{{\alpha} -1}\alpha( \alpha-1)^{\alpha-1}} \right).$$

In particular, if $f: \RR\to \RR$ is a $1$-Lipschitz function, and $l,m \in \{1,...,N\}$, $ l\leq m$, then for any $t>0$,
$$\PP^N_{V} \left( \frac{1}{N}\sum_{i=l}^m f\left(\overline{\lambda_i}\right) -  \frac{1}{N}\EE^N_{V}\sum_{i=l}^m  f\left(\overline{\lambda_i}\right) > t \right) \leq\exp\left( -\frac{bN^{2}t^{\alpha}}{2^{\alpha-1}\alpha(\alpha-1)^{\alpha-1} } \right),$$
where $\overline{\lambda_1},...,\overline{\lambda_N}$ is the rearrangement of the $\lambda_i$'s in ascending order.

\end{Pro}

\begin{proof}
Let 
$$ \forall \lambda \in \RR^N, \ \Phi(\lambda) = N\sum_{i=1}^N V(\lambda_i) -\frac{\beta}{2} \sum_{i\neq j} \log \left|\lambda_i - \lambda_j \right|.$$
We claim that $\Phi$ is $\alpha$-convex with respect to the norm $|| \ ||_{\alpha}$ on $\RR^N$, more precisely we will show that for all $\lambda,\mu \in \RR^N,$
\begin{equation} \label{alphaconv} \Phi(\lambda) + \Phi(\mu) -2\Phi\left( \frac{\lambda+\mu}{2} \right) \geq \frac{bN}{2^{\alpha-1}}  ||\lambda-\mu ||_{\alpha}^{\alpha}.\end{equation}
Note that for any $k,l \in \{1,...,N\}$,
$$ \mathrm{Hess}\left( -\beta \sum_{i\neq j} \log \left|\lambda_i - \lambda_j \right| \right)_{k,l} =
\begin{cases}
-\left( \lambda_k - \lambda_l\right)^{-2} & \text{ if } k\neq l,\\
\sum_{j\neq k} \left( \lambda_j - \lambda_k\right)^{-2} & \text{ if } k= l,
\end{cases}
$$
which defines a non-negative matrix since for any $x \in \RR^N$,
$$ \sum_{k\neq l } \left( \lambda_k - \lambda_l\right)^{-2} x_k^2 -  \sum_{k\neq l}  \left( \lambda_k - \lambda_l\right)^{-2}x_kx_l = \sum_{k<l} \left( \lambda_k - \lambda_l\right)^{-2} \left( x_k -x_l\right)^2\geq 0.$$
As by assumption
$$ \forall x \in \RR, \ V(x) = b|x|^{\alpha} + w(x),$$
with $w$ a convex function, we have, with the above observation, for any $\lambda, \mu \in \RR^N$,
$$ \Phi(\lambda) + \Phi(\mu) -2\Phi\left( \frac{\lambda+\mu}{2} \right) \geq bN\left( \sum_{i=1}^N \lambda_i^{\alpha} +\sum_{i=1}^N \mu_i^{\alpha} - 2 \sum_{i=1}^N \left( \frac{\lambda_i+\mu_i}{2} \right)^{\alpha} \right).$$
Since $\alpha\geq 2$, we have  for any $x,y\in \RR$,
$$ \frac{1}{2} x^{\alpha}  +\frac{1}{2}y^{\alpha}  \geq \left(\frac{x+y}{2} \right)^{\alpha}  + \left( \frac{x-y}{2} \right)^{\alpha}.$$
This yields the desired inequality \eqref{alphaconv}.

We know, by \cite[Corollary 4.1]{Bobkov}, that \eqref{alphaconv} entails that for any $1$-Lipschitz function with respect to $||\ ||_{\alpha}$, $g : \RR^N \to \RR$, and every $t>0$,
\begin{equation} \label{conc} \PP_{V}^N\left( g- \EE_{V}^N g> t\right) \leq  \exp\left( -\frac{bNt^{\alpha} }{2^{{\alpha} -1}\alpha( \alpha-1)^{\alpha-1}} \right).\end{equation}

Let now $f: \RR\to \RR$ be a $1$-Lipschitz function, and $k,l\in \{1,...,N\}$, $k\leq l$. We set
$$ \forall \lambda \in \RR^N, \ g(\lambda) = \frac{1}{N}\sum_{i=l}^m f(\overline{\lambda_i}),$$
For any $\lambda, \mu \in \RR^N$, we have by Cauchy-Schwartz inequality
$$
g(\lambda) - g(\mu) \leq  \frac{1}{N} \sum_{i=l}^m \left|\overline{\lambda_i} -\overline{ \mu_i}\right| \leq \frac{1}{N^{1/2}} \left(\sum_{i=l}^m |\overline{\lambda_i} - \overline{\mu_i}|^2 \right)^{1/2}\leq \frac{1}{N^{1/2}}\left( \sum_{i=1}^N \left|\lambda_i - \mu_i\right|^2\right)^{1/2},$$
where we used in the last inequality Hardy-Littlewood-Poly\'a rearrangement inequality.
Thus, by Hölder inequality
$$g(\lambda) - g(\mu) \leq N^{-\frac{1}{\alpha}} ||\lambda -\mu||_{\alpha}.$$
This shows that $g$ is $N^{-\frac{1}{\alpha}}$-Lipschitz with respect to the norm $|| \ ||_{\alpha}$. Applying Proposition \ref{conc}  to $g$ gives the second inequality in the statement.
\end{proof}

In the following proposition, we use the concentration inequalities of Proposition \ref{conckconv}, together with a truncation procedure and the deviations estimate of Proposition \ref{probatrou}, to prove that 
$$\frac{1}{N} \sum_{i=\log N+1}^{N}  {\lambda_i^*}^p ,$$
is exponentially equivalent to its expectation with respect to $\PP^N_{V}$. Combining this with the result of Proposition \ref{convmomenttronq}, 
$$\frac{1}{N} \sum_{i=\log N+1}^{N}\EE^N_{V}  {\lambda_i^*}^p \underset{N\to +\infty}{\longrightarrow} \langle\sigma_{sc}, x^p \rangle,$$
we will get Proposition \ref{expoequiv}.

\begin{Pro}\label{concent}
For any $t>0$,
$$ \limsup_{N\to +\infty} \frac{1}{N^{1+\alpha/p}} \log \PP^N_{V}\left( \bigg| \sum_{i=\log N+1}^{N}  {\lambda_i^*}^p - \sum_{i=\log N+1}^{N} \EE_{V}^N {\lambda_i^*}^p \bigg| \geq tN \right) = -\infty,$$
where $\lambda_1^*,...,\lambda_N^*$ denotes the rearrangement of the $\lambda_i$'s by decreasing absolute values.

\end{Pro}

\begin{proof} To ease the notation, we set $k =\log N$. The first part of the argument consists in choosing the proper truncation level with respect to our exponential scale $N^{1+\alpha/p}$.  For any $M_0>0$, we denote by $F_{M_0}$ the function
$$ \forall x \in \RR,  \ F_{M_0} (x) = 
\begin{cases}
\sg(x)\left(|x|\wedge M_0\right)^p & \text{ if $p$ is odd,}\\
\left(|x|\wedge M_0\right)^p & \text{ if $p$ is even.}
\end{cases}.$$
Let 
$$M_0= \frac{N^{\frac{1}{\alpha(p-1)}(1-\frac{\alpha}{p})}}{(\log N)^{\frac{1}{\alpha}}}.$$
Note that,
\begin{align*}
\left| \frac{1}{N}\sum_{i=k+1}^{N} \EE_{V}^N\left( {\lambda_i^*}^p -  F_{M_0}\left( \lambda_i^*\right) \right)\right | & 
\leq \frac{1}{N}\sum_{i=k+1}^{N}   \EE_{V}^N \left|\lambda_i^* \right|^p\Car_{|\lambda_i^*|\geq M_0} \\
& \leq \frac{1}{NM_0}\sum_{i=k+1}^{N}   \EE_{V}^N \left|\lambda_i^* \right|^{p+1}\\
& \leq \frac{N-k}{NM_0}   \EE_{V}^N \left|\lambda_1^* \right|^{p+1}\underset{N\to +\infty}{\longrightarrow} 0,
\end{align*}
using \eqref{momentborne}, and the fact that as $p>\alpha$, $M_0 \to +\infty$.
Thus, it suffices to prove that for any $t>0$,
$$ \limsup_{N\to +\infty} \frac{1}{N^{1+\alpha/p}} \log \PP^N_{V}\Big( \Big| \sum_{i=k+1}^{N}  {\lambda_i^*}^p -\sum_{i=k+1}^{N}\EE_{V}^N F_{M_0}\left(\lambda_i^*\right) \Big| \geq tN \Big) = -\infty.$$
 Note that, 
$$ \sum_{i=k+1}^{N} F_{M_0}\left({\lambda_i^*}\right) = \sum_{j=(k-l)+1 }^{N-l} F_{M_0}\left( \overline{\lambda_i} \right),$$
where $ l = \Card\{ i\in \{1,...,k\} : \lambda_i^* > 0 \}$. Since the function $F_{M_0}$ is $pM_0^{p-1}$-Lipschitz, we have using a union bound and Proposition \ref{conckconv}, for any $t>0$,
\begin{equation*} \PP_{V}^N\Big( \Big| \sum_{i=k+1}^{N}F_{M_0} \left({\lambda_i^*}\right) - \sum_{i=k+1}^{N}\EE_{V}^N F_{M_0}\left(\lambda_i^*\right)\Big|>tN \Big) \leq 2k\exp\left( -\frac{1}{c_{\alpha}p^{\alpha}} t^{\alpha}N^{1+\frac{\alpha}{p}}\log N\right),\end{equation*}
where $c_{\alpha}$ is some constant depending on $\alpha$.
We can write,
\begin{align*}
\PP_{V}^N&\Big( \Big| \sum_{i=k+1}^{N} {\lambda_i^*}^p -\sum_{i=k+1}^{N} \EE^N_{V} F_{M_0}(  {\lambda_i^*} ) \Big| >Nt\Big) \\
 &\leq \PP_{V}^N\Big( \Big| \sum_{i=k+1}^{N} F_{M_0}\left(  {\lambda_i^*} \right) - \sum_{i=k+1}^{N}\EE^N_{V} F_{M_0}(  {\lambda_i^*}) \Big|> tN/2\Big) \\
&+\PP^N_{V} \Big( \sum_{i=k+1}^{N} | \lambda_i^* |^p\Car_{M_0\leq |\lambda_i^*|} > tN/2 \Big).
\end{align*}
We saw by the concentration inequality above, that the deviations of the truncated moments at the level $M_0$ around its mean are exponentially negligible at the scale $N^{1+\alpha/p}$. We need now to prove that the contributions in the deviations of the truncated moments of the $\lambda_i$'s above the level $M_0$ are also negligible. To do so, we will truncate one more time at a level $R$, chosen so that the deviation bound of Proposition \ref{probatrou} gives the right exponential estimate. 

From \eqref{assumpV}, we  have for $M$ large enough,
$$ \inf_{|x| \geq M} V(x) \geq \frac{b}{2}M^{\alpha}.$$
 Proposition \ref{probatrou} yields that there are some constants $M_0>0$, and $C>0$, depending on $V$ and $\beta$, such that for any $M>M_0$, and $k \in \{1,...,N\}$,
\begin{equation} \label{probadev}\PP_{V}^N\left( L_N(I_M^c)  \geq \frac{k}{N} \right) \leq \exp\left( -CkNM^{\alpha}\right).\end{equation}
Let $R = e^{-1}\frac{N^{1/p}}{(\log N)^{1/2\alpha}}$. We have, with the inequality above, for $N$ large enough,
\begin{align}
\PP^N_{V} \Big( \sum_{i=k+1}^{N} | \lambda_i^* |^p\Car_{M_0\leq \left|\lambda_i^* \right|} > tN/2 \Big)& \leq 
\PP_{V}^N\Big( \sum_{i=k+1}^{N} \left| \lambda_i^*\right|^p\Car_{M_0\leq \left|\lambda_i^* \right|\leq R} > tN/ 2 \Big) \nonumber\\
&+ \PP_{V}^N\left( L_N\left( I_{R}^c \right)\geq \frac{k}{N} \right), \label{div}
\end{align}
where $L_N$ denotes the empirical measure of the $\lambda_i$'s, and where
 $I_{R} = [-R, R]$. From \eqref{probadev}, we deduce that,
$$ \PP_{V}^N\Big( L( I_{R}^c)\geq  \frac{k}{N}\Big) \leq \exp\Big( -Ce^{-\alpha}\left(\log N\right)^{1/2} N^{1+\frac{\alpha}{p}}\Big).$$
We are reduced to show that the event $\{ \sum_{i=k+1}^{N} \left| \lambda_i^*\right|^p\Car_{M_0\leq \left|\lambda_i^* \right|\leq R} > tN/ 2\}$ is exponentially negligible at the scale $N^{1+\alpha/p}$. To this end, we will slice up the set $\{ \lambda \in \RR : M_0 \leq | \lambda | \leq R\}$ into $\log\log N$ small intervals $\{ \lambda \in \RR : M_l\leq |\lambda|\leq M_{l+1}\}$ for which we will use the deviation bound \eqref{probadev}. At each step, we choose the largest bound so that the event $\{\sum_{i=k+1}^{N} \left| \lambda_i^* \right|^p\Car_{M_l\leq \left|\lambda_i^*\right|\leq M_{l+1}} >\frac{tN}{2}\}$ is exponentially negligible by \eqref{probadev}. 
For any $n \geq 1$, we set 
 $$q_n = \left(1-\frac{\alpha}{p}\right)\left(\frac{1}{p} + \frac{\alpha}{p^2} +...+ \frac{\alpha^{n-1}}{p^{n}} + \frac{\alpha^{n-1}}{p^n(p-1)} \right),$$
and 
$$M_n = \frac{N^{q_n}}{(\log N)^{1/\alpha}},$$

Observe that $q_n \underset{n\to +\infty}{\longrightarrow} \frac{1}{p}$, and $$\frac{1}{p} - q_n = O\Big(\left(\frac{\alpha}{p}\right)^p\Big).$$
Let $ n = \lfloor c\log  \log N \rfloor$ with $c$ such that $q_n \geq \frac{1}{p} - \frac{1}{\log N}$. With this choice, we have 
$$M_n\geq R.$$
Thus, slicing up the set $\{\lambda \in \RR : M_0 \leq |\lambda| \leq R\}$, we get
\begin{align*}
\PP_{V}^N\left( \sum_{i=k+1}^{N} \left| \lambda_i^* \right|^p\Car_{M_0\leq \left|\lambda_i^*\right|\leq R} > \frac{tN}{2} \right)& \leq \PP^N_{V}\left( \sum_{i=k+1}^{N} \left| \lambda_i^*\right|^p\Car_{M_0\leq \left|\lambda_i^*\right|\leq M_n} > \frac{tN}{2}\right)\\
& \leq  \PP_{V}^N\left( \sum_{l=0}^{n-1}\sum_{i=k+1}^{N} \left| \lambda_i^* \right|^p\Car_{M_l\leq \left|\lambda_i^*\right|\leq M_{l+1}} >\frac{tN}{2} \right)\\
& \leq  \PP_{V}^N\left( \sum_{l=0}^{n-1}M_{l+1}^p L_N\left(  I_{M_l}^c \right)>\frac{t}{2}\right).
\end{align*}
Finally, a union bound gives
$$\PP_{V}^N\left( \sum_{i=k+1}^{N} \left| \lambda_i^* \right|^p\Car_{M_0\leq \left|\lambda_i^*\right|\leq R_N} > \frac{tN}{2} \right) \leq \sum_{l=0}^{n-1} \PP^N_{V}\left(  L_N\left(  I_{M_l}^c \right)>\frac{t}{2nM_{l+1}^p} \right).$$
Using \eqref{probadev}, we get $N$ large enough, and for all $0\leq l\leq n$,
$$ \PP_{V}^N\Big(  L(  I_{M_l}^c)>\frac{t}{2nM_{l+1}^p} \Big) \leq \exp\Big( -\frac{CtN^2 M_l^{\alpha}}{2nM_{l+1}^p} \Big) \leq  \exp\left( -\frac{CtN^{2+\alpha q_l-pq_{l+1}} \left(\log N\right)^{\frac{p}{\alpha}-1}}{2c \log \log N} \right) .$$
But 
\begin{align*}
 \alpha q_l -pq_{l+1}& = \left( 1-\frac{\alpha}{p} \right) \left( \frac{\alpha}{p} +\frac{\alpha^2}{p^2} +... + \frac{\alpha^l}{p^l} + \frac{ \alpha^l}{p^l\left(p-1\right)}\right) \\
&- \left( 1-\frac{\alpha}{p} \right) \left( 1+\frac{\alpha}{p} +\frac{\alpha^2}{p^2} +... + \frac{\alpha^l}{p^l} + \frac{\alpha^l}{p^l\left(p-1\right)}\right) \\
&= -\left(1-\frac{\alpha}{p}\right).
\end{align*}
Therefore,
$$ \PP_{V}^N\Big(  L(I_{M_l}^c)>\frac{t}{2nM_{l+1}^p} \Big) \leq \exp\left( -\frac{CtN^{1+\frac{\alpha}{p}} \left(\log N\right)^{\kappa}}{2c \log \log N} \right),$$
where $\kappa >0$ as $p> \alpha$.
We can conclude that, 
$$\PP_{V}^N\left( \sum_{i=k+1}^{N} \left| \lambda_i^* \right|^p\Car_{M_0\leq \left|\lambda_i^*\right|\leq R_N} >tN/2 \right)\leq c\log \log N \exp\left( -\frac{CtN^{1+\frac{\alpha}{p}} \left(\log N\right)^{\kappa}}{2c \log \log N} \right),$$
which ends the proof.
\end{proof}

\subsection{Large deviations principle for the truncated moments}
Since we know from Proposition \ref{expoequiv}, that $(m_{p,N})_{N\in \NN}$ is exponentially equivalent to $$\left( \langle\sigma_{\beta}^V,x^p\rangle + T_{p,N}\right)_{N\in \NN},$$ we only need to derive a large deviations principle for $(T_{p,N})_{N\in \NN}$, in order to get the large deviations principle of $(m_{p,N})_{N\in \NN}$ (see \cite{Zeitouni}[Theorem 4.2.13]).

\begin{Pro}\label{LDPtruncmoment}
Under the assumption of Theorem \ref{ldp1} and the notation of Proposition \ref{expoequiv}, the sequence $\left( T_{p,N}\right)_{N\in \NN}$ follows a LDP under the law $\PP^N_{V}$, with speed $N^{1+\frac{\alpha}{p}}$, and good rate function $I_p$. If $p$ is odd, $I_p$ is defined by
 $$\forall x \in \RR,\ I_p(x) = b|x|^{\alpha/p},$$
and if $p$ is even,
$$ \forall x \in \RR, \ I_p(x) = 
\begin{cases}
bx^{\alpha/p} & \text{ if } x\geq 0,\\
+\infty& \text{ otherwise.}
\end{cases}
$$

\end{Pro}

\begin{proof}
To ease the notation, we set in the following $k=\log N$.\\
\textbf{Exponential tightness.}
Let 
$$\forall \lambda \in \RR^N, \ g(\lambda) = \Big( \sum_{i=1}^k \left|\lambda_i^* \right|^p \Big)^{1/p}.$$
For $\lambda \in \RR^N$, we set $l= \Card\{ i \in \{1,...,k\} : \lambda_i^*>0 \}$. We can write $$g(\lambda) = \Big( \sum_{i=1}^{k-l} |\overline{\lambda_i}|^p + \sum_{i= N-l+1}^N \overline{\lambda_i}^p \Big)^{1/p},$$
where $\overline{\lambda_1},...,\overline{\lambda_N}$ is the rearrangement of the $\lambda_i$'s in ascending order. When $l$ is fixed, as $p\geq \alpha$, we see that $g$ is $1$-Lipschitz with the same argument as in the proof of Proposition \ref{conckconv}. Using a union bound, we get by Proposition \ref{conckconv}, for any $t>0$,
 \begin{equation*} \label{tensionconc} \PP^N_{V} \Big(\Big( \frac{1}{N} \sum_{i=1}^k \left|\lambda_i^* \right|^p \Big)^{1/p} - \EE^N_{V} \Big( \frac{1}{N} \sum_{i=1}^k \left|\lambda_i^* \right|^p \Big)^{1/p}  >t \Big) \leq k\exp\left( -\frac{bt^{\alpha} N^{1+\frac{\alpha}{p}}}{2^{\alpha-1}\alpha(\alpha-1)^{\alpha-1}} \right). \end{equation*}
Besides, by Jensen's inequality
$$
\EE^N_{V} \left( \frac{1}{N} \sum_{i=1}^k \left|\lambda_i^* \right|^p \right)^{1/p} \leq 
 \Big( \EE^N_{V}\frac{1}{N} \sum_{i=1}^k \left|\lambda_i^* \right|^p \Big)^{1/p}
 \leq \left( \frac{k}{N}\EE^N_{V}  \left|\lambda_1^* \right|^p  \right)^{1/p}.
$$
From \eqref{momentborne}, we deduce
$$\EE^N_{V} \left( \frac{k}{N} \sum_{i=1}^k \left|\lambda_i^* \right|^p \right)^{1/p} \underset{N\to +\infty}{\longrightarrow} 0.$$
From the above concentration inequality, we see that $\left( T_{p,N}\right)_{N\in \NN}$ is exponentially tight.

\textbf{Upper bound.} Observe that we only have to show that for any $x>0$,
\begin{equation} \label{inequpp}\limsup_{N\to +\infty } \frac{1}{N^{1+\alpha/p}} \log \PP^N_{V}\left( T_{p, N} \geq x\right)\leq -I_p(x).\end{equation}
In the case where $p$ is even, it is clear that \eqref{inequpp}, is sufficient. In the case $p$ is odd,
observe that $\tilde{V}(x) = V(-x)$ satisfies the assumptions of Theorem \ref{ldp1}.  Note also that for any $x>0$,
$$ \PP^N_{V}\left(  T_{p, N} \leq -x \right) = \PP^N_{\tilde{V}}\left(  T_{p, N} \geq x \right).$$
Therefore, if \eqref{inequpp} is proven, and if $p$ odd, then we have also for any $x>0$,
$$\limsup_{N\to +\infty } \frac{1}{N^{1+\alpha/p}} \log \PP^N_{V}\left( T_{p, N} \leq -x \right) \leq -I_p(-x).$$

We now prove \eqref{inequpp}. Since $(\frac{1}{N}\sum_{i=1}^k |\lambda_i^*|^p)_{N\in \NN}$ is exponentially tight, we only need to show that for any $M>x>0$, we have 
 $$\limsup_{N\to +\infty } \frac{1}{N^{1+\alpha/p}} \log \PP^N_{V}\left( T_{p, N} \geq x, |\lambda_1^*|^p\leq MN \right)\leq -I_p(x).$$
Let $M>x>0$. Since the event $\{ T_{p, N} \geq x, |\lambda_1^*|^p\leq MN  \}$ is invariant under permutation of the $\lambda_i$'s, we have
\begin{align*}
&\PP^N_{V}\left( T_{p, N} \geq x, |\lambda_1^*|^p\leq MN\right)\\
& = \frac{N !}{Z^N_{V}} \int_{\underset{|\lambda_N|\leq ...\leq |\lambda_1|\leq (MN)^{1/p}}{ \sum_{i=1}^{k}  \lambda_{i}^p \geq Nx}} e^{-N\sum_{i=1}^N V\left( \lambda_i\right)} \prod_{i<j} \left|\lambda_i-\lambda_j\right|^{\beta} \prod_{i=1}^N d \lambda_i.
\end{align*}
Bounding the interaction term involving the $k$ largest in absolute value $\lambda_i$'s, we get
\begin{align*}
&\PP^N_{V}\left( T_{p, N} \geq x, |\lambda_1^*|^p\leq MN \right)\\
& \leq \frac{N !}{\left(N-k\right) !} \frac{Z^{N-k}_{\frac{NV}{N-k}}}{Z^N_{V}}  \Big( 2\left( NM\right)^{1/p}\Big)^{\beta Nk}\int_{\underset{|\lambda_{k}|\leq ...\leq |\lambda_1|}{ \sum_{i=1}^{k}  \lambda_{i}^p \geq Nx}} e^{-N\sum_{i=1}^{k} V\left( \lambda_i\right)}   \prod_{i=1}^{k} d \lambda_{i}\\
& \leq  \binom{N}{k}  \frac{Z^{N-k}_{\frac{NV}{N-k}}}{Z^N_{V}} \left( 2\left( NM\right)^{1/p}\right)^{\beta Nk} \int_{\sum_{i=1}^{k}\lambda_i^p\geq Nx } e^{-N\sum_{i=1}^{k} V\left( \lambda_i\right) } \prod_{i=1}^{k} d\lambda_i\\
& = \binom{N}{k}  \frac{Z^{N-k}_{\frac{NV}{N-k}}}{Z^N_{V}} \left( 2\left( NM\right)^{1/p}\right)^{\beta Nk} \left( \int e^{-NV(\lambda)} d\lambda \right)^{k} \PP\Big( \frac{1}{N} \sum_{i=1}^{k} X_i^p \geq x \Big),
\end{align*}
where $X_1,...,X_{k}$ are independent and identically distributed random variables with law $d \mu_V = e^{-NV(x)}\frac{dx}{Z_N}$, where $Z_N = \int e^{-NV(x)} dx$.
As 
$$ \int e^{-NV(x)} dx =e^{O(N)}, \text{ and } \log  \frac{Z^{N-k}_{\frac{NV}{N-k}}}{Z^N_{V}} = O\left( N\log N \right),$$
from Lemma \ref{encadrefuncpart} (recall that $k=\log N$), it only remains to show that
$$ \limsup_{N\to +\infty} \frac{1}{N^{1+\alpha/p}} \log   \PP\Big( \frac{1}{N} \sum_{i=1}^{\log N} X_i^p \geq x \Big) \leq -I_p(x).$$
This is the object of the following lemma.
\begin{Lem} \label{ldp2}Let $(X_j)_{j\geq 1}$ be a sequence of independent and identically distributed random variables with law $d \mu_V = e^{-NV(x)}\frac{dx}{Z_N}$, where $Z_N = \int e^{-NV(x)} dx$, with $V$ as in \eqref{assumpV}. Let $p\in \NN$, $p>\alpha$.

For any $x>0$,
$$\limsup_{N\to +\infty} \frac{1}{N^{1+\alpha/p}} \log   \PP\Big( \frac{1}{N} \sum_{i=1}^{\log N} X_i^p \geq x \Big) \leq -I_p(x),$$
with $I_p$ as in Proposition \ref{LDPtruncmoment}.
\end{Lem}

\begin{proof}
Let $x>0$. Set $Y_i = N^{1/\alpha}X_i$ for all $i\in \{1,...,\log N\}$. We have
$$
\PP\Big(\frac{1}{N} \sum_{i=1}^{\log N} X_i^p \geq x \Big)  \leq \PP\Big( \sum_{i=1}^{\log N} Y_i^p \geq   x N^{1+\frac{p}{\alpha}} \Big)
 \leq  \PP\Big( \sum_{i=1}^{\log N} \left|Y_i\right|^p \geq   x N^{1+\frac{p}{\alpha}} \Big).$$
 Let $0<t<1$. As $\alpha \leq  p$, we have $\alpha t /p <1$. Using the fact that $(x+y)^{s} \leq x^{s} + y^{s}$, for any $s\in (0,1)$, $x,y \in \RR^+$,
\begin{align*}
\PP\Big(\frac{1}{N} \sum_{i=1}^{\log N} X_i^p \geq x \Big)
 &\leq \PP\Big( \Big(\sum_{i=1}^{\log N} \left|Y_i\right|^p \Big)^{\frac{\alpha t}{p}} \geq x^{\frac{\alpha t}{p}} N^{t\left(1+\frac{\alpha}{p} \right)} \Big)\\
&  \leq \PP\Big( \sum_{i=1}^{\log N} \left|Y_i\right|^{\alpha t} \geq x^{\frac{\alpha t}{p}} N^{t\left(1+\frac{\alpha}{p}\right)} \Big).
\end{align*}
By Chernoff's inequality we get, 
\begin{equation}
\PP\Big(\frac{1}{N} \sum_{i=1}^{\log N} X_i^p \geq x \Big)\leq  e^{-bx^{\frac{\alpha t}{p}} N^{t\left(1+\frac{\alpha}{p} \right)}} \left(\EE\left( e^{b|Y_1|^{\alpha t}}\right)\right)^{\log N}. \label{chernoff}\end{equation}
As for any $x\in \RR$, $V(x) = b|x|^{\alpha} + w(x)$,
$$\EE\left( e^{b|Y_1|^{\alpha t}}\right) = \frac{1}{Z_N'} \int e^{-b(|x|^{\alpha} -\left|x\right|^{\alpha t}) -Nw\left(\frac{x}{N^{1/\alpha}}\right)} dx,$$
with 
$$Z_N '= \int e^{-NV\left(\frac{x}{N^{1/\alpha}}\right) }dx.$$
On one hand,
$$\int e^{-b(|x|^{\alpha} -\left|x\right|^{\alpha t}) -Nw\left(\frac{x}{N^{1/\alpha}}\right)} dx \leq 2e^{N \inf w} \int_0^{+\infty} e^{-b(x^{\alpha} -x^{\alpha t}) }dx.$$
Note that as $w$ is convex, $ \inf w > -\infty$.
On the other hand, $Z_N'= e^{O(N)}$.
Therefore,
$$
\EE\left( e^{b|Y_1|^{\alpha t}}\right)  \leq  e^{ o\left(\frac{ N^{1+\alpha/p}}{ \log N}\right)} \int_0^{+\infty} e^{-b(x^{\alpha} -x^{\alpha t}) }dx.$$
As $x \mapsto x^{\alpha-1} - t x^{\alpha t-1}$ is non-decreasing on $[1,+\infty)$, we have,
\begin{align*}
 \int_0^{+\infty} e^{-b\left(x^{\alpha} -x^{\alpha t}\right) }dx& \leq e^b + \frac{1}{\alpha(1-t)} \int_1^{+\infty} \left( \alpha x^{\alpha-1} -\alpha t x^{\alpha t-1}\right)e^{-b\left(x^{\alpha}-x^{\alpha t}\right)}dx \\
& = e^b +\frac{1}{b \alpha(1-t)}.
\end{align*}
Take $t=t_N = 1-1/(\log N)^2$. Then,
$$\EE\left( e^{b|Y_1|^{\alpha t_N}}\right) =e^{ o\left(\frac{ N^{1+\alpha/p}}{ \log N}\right)}.$$
Together with the bound \eqref{chernoff}, we get
\begin{equation} \label{bornesup}\frac{1}{N^{1+\alpha/p}} \log \PP\left(\frac{1}{N} \sum_{i=1}^{\log N} X_i^p \geq x \right) \leq -bx^{\frac{\alpha t_N}{p}} N^{-(1-t_N)\left( 1+\frac{\alpha}{p}\right)} + o(1).\end{equation}
Taking the limsup as $N$ goes to $+\infty$ we get the claim.
\end{proof} 
\textbf{Lower bound.}
Let $x \in \RR $. We want to show that
\begin{equation} \label{lowerboundbeta} \liminf_{\delta \to 0} \liminf_{N\to +\infty} \frac{1}{N^{1+\alpha/p}} \log \PP^N_V\left( T_{p,N} \in \left(x-\delta, x+\delta\right) \right) \geq -I_p(x).\end{equation}
As $T_{p,N}$ converges to $0$ in almost surely, it is enough to prove this bound for $x\neq 0$.
With the same argument as for the upper bound, it suffices actually prove to the bound above only for $x>0$. 

Let $x>0$ and $\delta>0$. We have for $N$ large enough,
$$\PP_{V}^N\left( T_{p,N} \in \left( x- \delta, x +\delta \right) \right) \geq \PP_{V,\beta}^N\left( \frac{1}{N}{\lambda_1^*}^p \in \left( x-\delta/2, x+\delta/2\right),  \forall  i>1, |\lambda_i^*| \leq M\right),$$
with $M>0$.
By continuity, there is some $\veps>0$ such that 
$$\PP_{V}^N\left( T_{p,N} \in \left( x- \delta, x +\delta \right) \right) \geq \PP_{V}^N\left( \frac{1}{N^{1/p}} \lambda_1^* \in \left( x^{1/p}-\veps, x^{1/p}+\veps\right), \forall  i>1, |\lambda_i^*| \leq M\right).$$
 We have
\begin{align*}
\PP_{V}^N&\left(T_{p,N} \in \left( x- \delta, x +\delta \right) \right)\\
& \geq N !\frac{Z^{N-1}_{\frac{NV}{N-1}}}{Z^N_{V}}  \int_{\big|\frac{\lambda_1}{N^{1/p}}-x^{1/p} \big|<\veps} d\lambda_1 e^{-NV\left(\lambda_1\right)} \EE^{N-1}_ {\frac{NV}{N-1}} \left( \Car_{L_{N-1}\in \mathcal{E}_M} e^{\beta\left(N-1\right)\langle \log \left(\lambda_1-.\right),L_{N-1}\rangle} \right),
\end{align*}
where $ L_{N-1} = \frac{1}{N-1} \sum_{i=2}^N \delta_{\lambda_i}$,
and
$\mathcal{E}_M = \{ \mu \in \mathcal{M}_1\left(\RR\right) : \mathrm{supp}(\mu) \subset [-M, M] \}$, with $\mathcal{M}_1\left(\RR\right)$ the set of probability measures on $\RR$. Thus,
\begin{align*}
 \PP_{V}^N\left( T_{p,N} \in \left( x- \delta, x +\delta \right) \right)&\geq N ! \frac{Z^{N-1}_{\frac{NV}{N-1}}}{Z^N_{V}}    \int_{ \big|\frac{\lambda}{N^{1/p}}-x^{\frac{1}{p}}\big|<\veps} e^{-NV\left(\lambda\right)}  d\lambda\\
& \times e^{\beta\left(N-1\right) \log (N^{\frac{1}{p}}x^{\frac{1}{p}}-M-\veps)}\PP^{N-1}_{\frac{NV}{N-1}} \left( L_{N-1}\in  \mathcal{E}_M \right).
\end{align*}
As $w(y) =o_{\pm \infty}( |y|^{\alpha})$, we have 
\begin{align*}
  \int_{\big|\frac{\lambda}{N^{1/p}}-x^{\frac{1}{p}} \big|<\veps} e^{-NV\left(\lambda\right)}  d\lambda &\geq \int_{\big|\frac{\lambda}{N^{1/p}}-x^{\frac{1}{p}} \big|<\veps} e^{-(b+o(1))N\lambda^{\alpha}} dy\\
& = e^{-(b-o(\veps))N^{1+\frac{\alpha}{p}}x^{\frac{\alpha}{p}}} e^{o(N^{1+\frac{\alpha}{p}})}.
\end{align*}
Thus,
$$\PP_{V}^N\left( T_{p,N} \in \left( x- \delta, x +\delta \right) \right) \geq \frac{Z^{N-1}_{\frac{NV}{N-1}}}{Z^N_{V}}\PP^{N-1}_{\frac{NV}{N-1}} \left( L_{N-1}\in  \mathcal{E}_M \right) e^{-(b-o(\veps))N^{1+\frac{\alpha}{p}}x^{\frac{\alpha}{p}}} e^{o(N^{1+\frac{\alpha}{p}})}.$$
But from Lemma \ref{encadrefuncpart} we know that $\log \frac{Z^{N-1}_{\frac{NV}{N-1}}}{Z^N_{V}} = O\left( N\right)$.
Besides, by Proposition \ref{probatrou} (with $k=1$), we have for $M$ large enough,
$$\PP^{N-1}_{\frac{NV}{N-1}} \left( L_{N-1}\in  \mathcal{E}_M \right) \underset{N\to +\infty}{\longrightarrow} 1.$$
This concludes the proof of the lower bound \eqref{lowerboundbeta}.
\end{proof}
\newpage
\section{The case of Wigner matrices without Gaussian tails}\label{noGaussian}

We will give in this section a proof of Theorem \ref{LDPtraceWG}. The strategy followed is in the same spirit as the ones developed in \cite{Bordenave}, \cite{Groux} and \cite{LDPevNG}. We start by a heuristic argument to give a idea of the nature of the deviations of the moments, and of the speed of the deviations.

\subsection{Heuristics}
We show here how one can get the lower bound of the LDP without much effort. The main fact which makes the argument work is the following : if we add to a given Hermitian matrix a low rank Hermitian matrix with not too large operator norm, then the map $A \mapsto \tr_N A^p$ is almost linear. More precisely, we have the following lemma, whose proof is postpone at section \ref{third step}.

\begin{Lem}\label{decompotrace}
Let $p\geq 2$. Let $A$ and $C$ be two Hermitian matrices of size $N$. Assume that $C$ is of rank at most $r$. We have
$$ \left|\tr\left(A + C\right)^p - \tr A^p - \tr C^p \right| \leq 2^p r \max_{1\leq k\leq p-1} ||A||^k ||C||^{p-k},$$
where $|| \ ||$ denotes the operator norm.
\end{Lem}

  To make the argument clearer, let us assume $X$ has entries distributed according to the exponential law with parameter $b$. We restrict ourself to the case where $p$ is even. Let $\delta> 0$ and $\theta_N = (N\delta)^{1/p}$. Denoting $X_N^{(1,1)} = X_N-\frac{X_{1,1}}{\sqrt{N}}e_1 e_1^*$, where $e_1$ is the first coordinate vector of $\CC^N$, we have
\begin{align*}
\PP\left( \tr_NX_N^p \simeq C_{p/2} +\delta \right)
&\gtrsim \PP\Big(  \tr_N\left(X^{(1,1)}_N + \theta_N e_1 e_1^*\right)^p \simeq C_{p/2}+\delta, \ \frac{X_{1,1}}{\sqrt{N}} \simeq \theta_N \Big)\\
&\gtrsim \PP\Big(  \tr_N\left(X^{(1,1)}_N + \theta_N e_1 e_1^*\right)^p \simeq C_{p/2}+\delta, ||X^{(1,1)}_N|| \leq c\Big)\\
&\times \PP\Big( \frac{X_{1,1}}{\sqrt{N}} \simeq \theta_N \Big),
\end{align*}
with some $c>2$. As $||X_N^{(1,1)}-X_N|| \to 0$ in probability, and 
$$ ||X_N|| \underset{N\to +\infty}{\longrightarrow} 2,$$ in probability by \cite[Theorem 5.1]{Silverstein} (or \cite[Theorem 2.1.22, Exercise 2.1.27]{Guionnet}), we have
$$\PP\left( ||X_N^{(1,1)}|| \leq c \right) \underset{N\to +\infty}{\longrightarrow } 1.$$
By Lemma \ref{decompotrace}, we have
\begin{align*}
\PP\left( \tr_NX_N^p \simeq C_{p/2} +\delta \right)
&\gtrsim \PP\Big( \tr_N\left(X^{(1,1)}_N\right)^p + \frac{1}{N}\theta_N^p\simeq C_{p/2}+\delta,  ||X^{(1,1)}_N|| \leq c\Big)\\
&\times\PP\left( \frac{X_{1,1}}{\sqrt{N}} \simeq \theta_N \right)\\
 &\gtrsim\PP\Big( \tr_N\left(X^{(1,1)}_N\right)^p \simeq C_{p/2},  ||X^{(1,1)}_N|| \leq c\Big)\PP\Big( \frac{X_{1,1}}{\sqrt{N}} \simeq \theta_N \Big).
\end{align*}
Since $X_{1,1}$ has exponential law with parameter $1$, we have
$$ \PP\left(\frac{X_{1,1}}{N^{\frac{1}{2} + \frac{1}{p}} } \simeq \delta   \right) \simeq \exp\left(-bN^{\frac{1}{2} + \frac{1}{p}}\delta \right).$$
But $(\tr_N\left(X^{(1,1)}_N\right)^p)_{N\in \NN}$ converges to $C_{p/2}$ in probability, by Wigner's theorem (see  \cite[Lemmas 2.1.6, 2.1.7]{Guionnet}). Therefore,
$$\PP\left( \tr_NX_N^p \simeq x \right) \gtrsim \exp\left(-bN^{ \frac{1}{2}+\frac{1}{p}}\delta\right).$$
The same argument can also be carried out to get the second part of the lower bound, using the deformation
$$ \left(\begin{array}{cc}
0 & \theta_N\\
\theta_N & 0
\end{array}\right),$$
with $\theta_N = \left( \frac{\delta N}{2} \right)^{1/p}$.

\subsection{Outline of proof}\label{outlineWG}
 As suggested by the heuristic argument above, the deviations of $\tr_NX_N^p$ are due to finite rank deformations of $X_N$ with entries of order $N^{1/p}$. 
We decompose $X_N$ in the following way

\begin{equation}\label{cut} X_N = A +B^{\veps} + C^{\veps} + D^{\veps},\end{equation}
with 
$$ A_{i,j} = \frac{X_{i,j}}{\sqrt{N}}\Car_{ \left|X_{i,j}\right| \leq  (\log N)^d}, \quad \quad B^{\veps}_{i,j} =  \frac{X_{i,j}}{\sqrt{N}}\Car_{ (\log N)^d < \left|X_{i,j}\right|< \veps N^{\frac{1}{2}+\frac{1}{p}}},$$
$$ C^{\veps}_{i,j} = \frac{X_{i,j}}{\sqrt{N}}\Car_{\veps N^{\frac{1}{2}+\frac{1}{p}}\leq  \left|X_{i,j}\right|\leq \veps^{-1}N^{\frac{1}{2}+\frac{1}{p}}}, \quad \quad D^{\veps}_{i,j} =  \frac{X_{i,j}}{\sqrt{N}}\Car_{ \veps^{-1}N^{\frac{1}{2}+\frac{1}{p}}< \left|X_{i,j}\right| },$$
where where $d$ is taken such that $\alpha d >1$.

In a first phase, we will show that one can neglect in the deviations of $\tr_NX_N^p$ the contributions of the intermediate entries, that is $B^{\veps}$, and the large entries, that is $D^{\veps}$, so that $(\tr_N(A+C^{\veps})^p)_{N\in \NN,\veps>0}$ are exponentially good approximations for $(\tr_NX_N^p)_{N\in \NN}$. 

Then, due to concentration inequalities, we show that the conditional expectation given $C^{\veps}$, $\EE_{C^{\veps}}\tr_N(H+C^{\veps})^p$,  where $H$ is a copy of $A$ independent of $X$, are exponentially good approximations of $(\tr_NX_N^p)_{N\in \NN}$. From the choice of the decomposition \eqref{cut}, we deduce that $C^{\veps}$ has only a finite number of non-zero entries at the exponential scale $N^{1+\alpha/p}$. Thus, Lemma \ref{decompotrace} and  Wigner's theorem allow us to conclude that $(\EE_{C^{\veps}} \tr_N(H+C^{\veps})^p)_{N\in \NN}$ is exponentially equivalent to $(\langle \sigma_{sc}, x^p\rangle + \tr_N (C^{\veps})^p)_{N\in \NN}$. It only remains to show a large deviations principle $C^{\veps}$, and conclude by contraction principle, with an argument similar as in \cite{LDPevNG}. The use of the contraction principle is made possible by the fact that $C^{\veps}$ has a finite number of non-zero entries with exponentially large probability. 

\subsection{Concentration inequalities} \label{concentration inequalities}
In this section, we revisit a concentration inequality from \cite{Meckes} for the trace of powers of sum of a Hermitian matrix with bounded entries with a deterministic Hermitian matrix. This inequality will be crucial to get the exponential tightness and an exponential approximation of  $(\tr_N X_N^p)_{N\in \NN}$.

Unfortunately, we cannot directly use the concentration inequality of  \cite[Proposition 4]{Meckes}, because of the assumption made on the expectation of the entries. To make the strategy sketched in \ref{outlineWG} work, we need to prove a concentration inequality for 
$$ \tr_N\left( \frac{H}{\sqrt{N}}+\frac{C}{\sqrt{N}} \right)^p,$$
 where $H$ is a centered matrix with bounded entries, and where $C$ is a deterministic matrix whose entries are of order $N^{1/p+1/2}$. But then,
\begin{equation} \label{bornexpect} \tr \left( \frac{C}{\sqrt{N}}\right)^{2(p-1)} \leq r^{2(p-1)} N^{\frac{2(p-1)}{p}},\end{equation}
where $r$ is the number of non-zero entries of $C$, which is a bound too loose to use the concentration inequality of \cite[Proposition 4]{Meckes}. 

However, since we are considering normalized traces, we are looking at deviations of order $N$ of the traces, whereas in \cite{Meckes} the deviations considered were of order $1$. Thus, one can expect that there is some room left in the approach of Meckes and Szarek, to get a concentration inequality for $ \tr_N\left(H+C \right)^p$, with the bound \eqref{bornexpect}.
 
\begin{Pro}\label{conctr}Let $p\in \NN$, $p\geq 3$.
Let $H$ be a centered random Hermitian matrix such that $(H_{i,j})_{i\leq j}$ are independent and bounded by some $\kappa\geq 1$, and let $C$ be a deterministic Hermitian matrix such that $\tr ( \frac{C}{\sqrt{N}} )^{2(p-1)} \leq m N^{2-\frac{2}{p}}$, where $m\geq1$. There are some universal constants $c, c' >0$, such that for all $t \geq  c'(p  m^{p-1})^{p} N^{-\frac{1}{2} \left( 1+\frac{2}{p}\right)}$,
$$\PP\left( \left|\tr_N(H+C)^p - \EE\tr_N (H+C)^p \right|>tN^{p/2} \right ) \leq 8\exp\left(-\frac{N^{1+\frac{2}{p}}}{c \kappa^2}\min\Big\{\left( \frac{t}{p}\right)^{2/p} , \frac{t^2}{p^2m^{2(d-1)}}\Big\}\right).$$
Moreover,
$$\PP\left( \Big |\tr_N \left|H+C\right|^p - \EE\tr_N \left|H+C\right|^p \Big|>tN^{p/2} \right ) \leq 8\exp\left(-\frac{N^{1+\frac{2}{p}}}{c\kappa^2}\min\Big\{\left( \frac{t}{p}\right)^{2/p} , \frac{t^2}{m^{2(d-1)}}\Big\}\right).$$
\end{Pro}

\begin{proof}We follow the same approach as in \cite[Proposition 4]{Meckes}, with some slight variations at times, but considering deviations of order $N^{1+p/2}$ of the trace of $(H+C)^p$. We will prove only the first inequality, the proof of the second inequality being exactly the same.

Without loss of generality, we can assume $\kappa =1$. Let $X= H+C$. For $\beta \in \{1,2\}$, we denote by $\mathcal{H}_N^{(\beta)}$ the set of symmetric matrices of size $N$, when $\beta=1$, and Hermitian matrices when $\beta =2$. Note that as $H$ has entries bounded by $1$, we know by \cite[Corollary 4.10]{Ledouxconcentration}, that for any convex and $1$-Lipschitz function $f : \mathcal{H}^{(\beta)}_N \to \RR$ with respect to the Hilbert-Schmidt norm, and all $t>0$,
$$ \PP\left( \left| f(X) - \MM f(X) \right|>t \right) \leq 4 e^{ -\frac{t^2}{4}},$$
where $\MM f(X)$ denotes the median of $f(X)$. 
Let $a>0$. Define
$$K_a = \left\{ Y\in \mathcal{H}^{(\beta)}_n : ||Y||_{2(p-1)} \leq a \right\},$$
where $||Y||_{q} = (\tr|Y|^q)^{1/q}$ for any matrix $Y$ and $q >0$.
Note that we can write
$$ F = F^+ - F^-,$$
with $F^+(Y ) = \tr Y^p_{+}$, and $F^-(Y) = \tr Y_{-}^p$ for any $Y\in  \mathcal{H}^{(\beta)}_N$, where for every $x\in \RR$, $ x_+$ and $x_-$ denote the positive and negative parts of $x$. The functions $F^+$ and $F^-$ are convex and $pa^{p-1}$-Lipschitz on $K_a$.  Let $F_a^+$, $F_a^-$ denote the convex extensions of $F^+_{|K_a}$ and $F^-_{|K_a}$ to $ \mathcal{H}^{(\beta)}_N$, which are $pa^{p-1}$-Lipschitz, as explained in \cite[Lemma 5]{Meckes}. Then, for all $t>0$, we have
$$ \PP\left( \left|F_a^{\sigma}(X) - \MM F^{\sigma}_a(X) \right|>tN^{1+p/2} \right) \leq 4 \exp\left( -\frac{t^2N^{p+2}}{4 p^2 a^{2(d-1)} }\right),$$
with $\sigma \in \{+,-\}$. 

Besides $Y\mapsto ||Y||_{2(d-1)}$ is convex and $1$-Lipschitz with respect to the Hilbert-Schmidt norm. From \cite[Theorem 8.6]{Massart}, we deduce that for any $t>0$,
$$\PP\left(  ||X||_{2(p-1)} - \EE ||X||_{2(p-1)}>t \right)\leq e^{ -\frac{t^2}{32}}.$$ 
But, 
$$\EE||X||_{2(p-1)} \leq \EE||H||_{2(p-1)} + ||C||_{2(p-1)} \leq  N^{\frac{1}{2(p-1)}}\EE||H||+mN^{\frac{1}{2}+\frac{1}{p}},$$
where $||\ ||$ denotes the operator norm, and where we used the fact that $m\geq1$.
But we know from \cite[p.6]{Meckes}, that there is some universal constant $c_1\geq 1$, such that 
$$  \EE||H||  \leq c_1 \sqrt{N}.$$
Thus, $\EE||X||_{2(p-1)} \leq 2m c_1N^{\frac{1}{2} + \frac{1}{p}}$.

Let now $b>0$, and $a= b N^{\frac{1}{2}+\frac{1}{p}}$. We have, for $b\geq4 mc_1$,
$$\PP\Big(  ||X||_{2(p-1)}  \geq a \Big)\leq \PP\Big(  ||X||_{2(p-1)} - \EE ||X||_{2(p-1)}\geq\frac{a}{2} \Big)\leq \exp\left( -\frac{b^2N^{1+\frac{2}{p}}}{128}\right).$$ 
Besides, with this choice of $a$, we have for all $t>0$, and all $\sigma \in \{+,-\}$,
$$ \PP\Big( \left|F_a^{\sigma}(X) - \MM F_a^{\sigma}(X) \right|>\frac{t}{2}N^{1+p/2} \Big) \leq 4 \exp\left( -\frac{t^2N^{1+2/p}}{16 p^2 b^{2(p-1)} }\right).$$ 
Thus,
\begin{align}
\PP\Big( \left|F^{\sigma}(X) - \MM F^{\sigma}_a(X) \right|>\frac{t}{2} N^{1+p/2} \Big) &\leq \PP\Big(\left|F_a^{\sigma}(X) - \MM F^{\sigma}_a(X) \right|> \frac{t}{2} N^{1+p/2}\Big)\nonumber \\
&+\PP\Big(||X||_{2(p-1)} \geq a \Big) \nonumber\\
&  \leq 4 \exp\left( -\frac{N^{1+2/p}}{128 } \min\Big\{ b^2 , \frac{t^2}{p^2b^{2(p-1)} }\Big\}\right)\nonumber.
\end{align}
As a consequence, for $b=4mc_1$, we can find a numerical constant $c_2\geq 1$, such that for $t = c_2pN^{-\frac{1}{2} \left( 1+\frac{2}{p}\right)}$, we have
$$\PP\Big( F^{\sigma}(X) -\MM F^{\sigma}_a(X) > tN^{1+p/2} \Big) < \frac{1}{2}.$$
We deduce that 
$$ \MM F^{\sigma}(X) \leq \MM F^{\sigma}_{a}(X) +c_2pN^{\frac{1}{2} +\frac{p}{2}-\frac{1}{p} }.$$
As $F^{\sigma}_a$ is non-decreasing with $a$, and $F^{\sigma}_a\leq F^{\sigma}$ for any $a>0$, we have for all $b\geq 4mc_1$,
$$ \MM F^{\sigma}(X) -c_2pN^{\frac{1}{2} +\frac{p}{2}-\frac{1}{p} } \leq \MM F^{\sigma}_{a}(X) \leq \MM F^{\sigma}(X).$$
Thus, for $t \geq 2c_2p N^{-\frac{1}{2} \left( 1+\frac{2}{p}\right)}$, and any $b\geq 4mc_1$, we deduce that
\begin{align*}\PP\left( \left|F^{\sigma}(X) - \MM F^{\sigma}(X) \right|>t N^{1+p/2} \right) &
 \leq 
\PP\Big( \left|F^{\sigma}(X) - \MM F^{\sigma}_a(X) \right|>\frac{t}{2} N^{1+p/2} \Big)\nonumber\\
& \leq 4 \exp\left( -\frac{N^{1+2/p}}{128 } \min\Big\{ b^2 , \frac{t^2}{p^2b^{2(p-1)} }\Big\}\right). \end{align*}
But one can check that,
$$
\max_{b\geq 4 mc_1} \min\left\{ b^2 , \frac{t^2}{p^2b^{2(p-1)}} \right\} 
= \min\left\{ \left(\frac{t}{p}\right)^{2/p} , \frac{t^2}{p^2(mc_1)^{2(p-1)}}\right\}.
$$
Optimizing in $b$ in the previous inequality, and setting $c_3 =128 c_1^{2(p-1)}$, we get 
$$\PP\Big( \left|F^{\sigma}(X) - \MM F^{\sigma}(X) \right|>t N^{1+p/2} \Big) \leq 4\exp\left(-\frac{N^{1+\frac{2}{p}}}{c_3}\min\Big\{\left( \frac{t}{p}\right)^{2/p} , \frac{t^2}{p^2m^{2(p-1)}}\Big\}\right).$$
To get the same inequality but with $\EE F^{\sigma}(X)$ instead of $\MM F^{\sigma}(X)$, we integrate by parts the inequality above, and we find that there is some constant $c_4>0$, such that
$$\left|\EE F^{\sigma}(X)-\MM F^{\sigma}(X)\right|\leq c_4m^{p-1}p N^{-\frac{1}{2}\left( 1+\frac{2}{p}\right)}.$$
At the price of taking $c_4$ larger, we can assume that $c_4\geq c_2$. Then, for every $t \geq 2c_4m^{p-1}p N^{-\frac{1}{2}\left( 1+\frac{2}{p}\right)}$, 
\begin{align*}
\PP\Big( \left|F^{\sigma}(X) - \EE F^{\sigma}(X) \right|>t N^{1+p/2} \Big) &\leq \PP\Big( \left|F^{\sigma}(X) - \MM F^{\sigma}(X) \right|>\frac{t }{2}N^{1+p/2} \Big)\\
& \leq 4\exp\left(-\frac{N^{1+\frac{2}{p}}}{4c_3}\min\Big\{\left(\frac{t}{p}\right)^{2/p} , \frac{t^2}{p^2m^{2(d-1)}}\Big\}\right).
\end{align*}
As $F=F^+-F^-$, we have for any $t \geq 2c_4 m^{p-1}pN^{-\frac{1}{2}\left( 1+\frac{2}{p}\right)}$,
$$\PP\left( \left|F(X) - \EE F(X) \right|>t N^{1+p/2} \right)\leq 8\exp\left(-\frac{N^{1+\frac{2}{p}}}{16c_3}\min\left\{\left(\frac{t}{p}\right)^{2/p} , \frac{t^2}{m^{2(d-1)}}\right\}\right).$$
Setting $c = 16c_3$, and $c' = 2c_4$, we get the claim.
\end{proof}
\subsection{Exponential tightness}
Throughout the rest of this section, we fix a constant $\gamma>0$, such that for $t$ large enough,
\begin{equation} \label{tail} \PP\left( \left|X_{1,1}\right| >t \right)\vee \PP\left( \left|X_{1,2}\right| >t \right)  \leq e^{-\gamma t^{\alpha}}.\end{equation}

In this section, we will show that the sequence $(\tr_NX_N^p)_{N\in \NN }$ is exponentially tight, namely, we have the following proposition.
\begin{Pro}[Exponential tightness]\label{tensionexpo}
$$\lim_{t\to +\infty} \bornesupalpha  \PP\left( \tr_N\left|X_N \right|^p>t \right) = -\infty.$$
\end{Pro}

\begin{proof}[Proof of Proposition \ref{tensionexpo}]
Using the triangular inequality for the $p$-Schatten norm, we get for any $t>0$,
\begin{align}\label{decouptr}
\PP\left(\tr_N\left|X_N \right|^p>(4t)^{p} \right) &\leq \PP\left( \tr_N\left|A\right|^p>t^{p} \right) + \PP\left( \tr_N\left|B^{\veps} \right|^p>t^{p} \right)\nonumber\\
&+ \PP\left(\tr_N\left|C^{\veps}\right|^p>t^{p} \right)+\PP\left( \tr_N\left|D^{\veps} \right|^p>t^{p} \right).
\end{align}
This shows that it suffices to estimate at the exponential scale, the probability of each event $\{\tr_N\left|A\right|^p>t^{p}\}$, $\{\tr_N\left|B^{\veps} \right|^p>t^{p} \}$, $\{\tr_N\left|C^{\veps}\right|^p>t^{p}\}$, and finally $\{\tr_N\left|D^{\veps} \right|^p>t^{p} \}$.
As a consequence of the concentration inequality of Proposition \ref{conctr}, we have the following lemma.

\begin{Lem}\label{tensionexpoA}
$$\lim_{t \to +\infty}\bornesup \PP\left(\tr_N\left|A\right|^p >t \right) =-\infty,$$
where $A$ is as in $\eqref{cut}$.
\end{Lem}

\begin{proof}
Note that as $p\geq 2$,
$$\tr\Big( \EE A\Big)^{2(p-1)} \leq \left(\tr( \EE A )^2 \right)^{p-1}.$$
Since the entries of $X$ are centered, we get
$$\tr \Big( \EE A\Big)^2 = \frac{1}{N}\sum_{1\leq i, j\leq N} \EE |X_{i,j}|^2\Car_{|X_{i,j}| > (\log N)^d}.$$
Integrating by parts, we have
 $$\tr \Big( \EE A\Big)^2 = O\Big(N^2 e^{-\frac{\gamma}{2} (\log N)^{\alpha d}}\Big),$$
where $\gamma$ is as in \eqref{tail}. As $\alpha d> 1$,
\begin{equation}\label{momentA}\tr\Big(\EE A\Big)^{2(p-1)} =o (1).\end{equation}
We see that $A$ satisfies the assumptions of Proposition \ref{conctr} with some $m\geq1 $ and $\kappa = (\log N)^d$. We get for any $t>0$, and $N$ large enough,
$$\PP\left( \left|\tr_N \left|A\right|^p - \EE\tr_N \left|A\right|^p \right|>t\right ) \leq 8\exp\left(-\frac{N^{1+\frac{2}{p}} }{c p^2(\log N)^{2d} }\min\Big\{t^{2/p} , \frac{t^2}{m^{2(p-1)}}\Big\}\right),$$
which yields, as $\alpha <2$,
\begin{equation} \label{concAexpo}\bornexpoalpha \PP\left( \left|\tr_N \left|A\right|^p - \EE\tr_N \left|A\right|^p \right|>t\right ) = -\infty.\end{equation}
We know from \cite[Theorem 2.1.1, Lemma 2.1.6]{Guionnet}, that 
\begin{equation}\label{convmomentabs} \EE \tr_N |X_N|^p \underset{N\to +\infty}{\longrightarrow} \langle \sigma_{sc}, |x|^p\rangle,\end{equation}
where $\langle\sigma_{sc}, |x|^p\rangle = \int |x|^p d\sigma_{sc}(x)$.
Denoting $\mu_{X_N}$ and $\mu_A$ the spectral measures of $X_N$ and $A$ respectively, we have using the decreasing coupling and \cite[Theorem III 4.4]{Bhatia},
\begin{equation} \label{inegWp} W_p(\EE \mu_{X_N}, \EE \mu_A) \leq \Big(\EE \tr_N|X_N-A|^p\Big)^{1/p},\end{equation}
where $W_p$ denotes the $p$-Wasserstein distance. As a consequence of the polar decomposition, we can write $|X_N-A|^p = (X_N-A)^pU$, where $U$ is  a unitary matrix, so that
\begin{equation}\label{polar}
\EE \tr|X_N-A|^p\leq \frac{1}{N^{p/2}}\sum_{i_1,...,i_{p+1}}\EE \prod_{j=1}^p |X_{i_j,i_{j+1}}|\Car_{|X_{i_j,i_{j+1}}| \leq (\log N)^d},\end{equation}
Hölder inequality yields,
$$
\EE \tr|X_N-A|^p\leq N^{p/2+1} \max\Big( \EE|X_{1,1}|^p\Car_{|X_{1,1}|> (\log N)^p} , \EE|X_{1,2}|^p\Car_{|X_{1,2}|> (\log N)^p}\Big),$$
where we used the fact that the entries of $X$ are centered. 
Integrating by parts, we get
\begin{equation} \label{diffXA} \EE \tr|X_N-A|^p = O\Big(N^{p/2+1} e^{-\frac{\gamma}{2}(\log N)^{\alpha d} }\Big),\end{equation}
where $\gamma$ is as in \eqref{tail}.
As $\alpha d>1$, we deduce by \eqref{inegWp}, $W_p(\EE \mu_{X_N}, \EE \mu_A) =o(1)$, which yields
$$ \Big| \EE \tr_N |X_N|^{p} - \EE \tr_N |A|^{p} \Big| =  o(1).$$
We can conclude with \eqref{concAexpo} and \eqref{convmomentabs} that $(\tr_N |A|^p)_{N\in \NN}$ is exponentially tight.

\end{proof}

For the second event $\{\tr_N |B^{\veps}|^p >t^p \}$, we have the following lemma.
\begin{Lem}\label{tensionexpoB}
For any $\veps>0$, we have
$$\lim_{t \to +\infty} \bornesupalpha \PP\left(  \tr_N\left|B^{\veps}\right|^p >t \right) = -\infty.$$
\end{Lem}

\begin{proof}
Since $p\geq 2$, we have
$$ \left(\tr \left|B^{\veps}\right|^p \right)^{2/p}\leq \tr (B^{\veps})^2.$$
Thus,
$$ \PP\left( \tr |B^{\veps}|^p \geq t N \right) \leq \PP\left( \tr (B^{\veps})^2 \geq t^{2/p} N^{2/p} \right).$$
Chernoff's inequality yields for any $\lambda>0$,
$$ \PP\Big( \sum_{1\leq i\leq j\leq N}\big| B^{\veps}_{i,j}\big|^2 \geq \frac{t^{2/p } }{2} N^{2/p} \Big)  \leq e^{-\frac{\lambda}{2} t^{\frac{2}{p}} N^{\frac{2}{p}+1}} \prod_{1\leq i\leq j\leq N} \EE\left( e^{\lambda |X_{i,j}|^2 \Car_{(\log N)^d < |X_{i,j}|< \veps N^{\frac{1}{2} + \frac{1}{p}} } } \right).$$
Let $1\leq i \leq j \leq N$. Recall that for $\mu$ a probability measure on $\RR$ and $g\in C^1$, we have the following integration by parts formula :
$$ \int_a^b g(x) d\mu(x) = g(a)\mu\left[a, +\infty\right) - g(b) \mu\left(b, +\infty\right) + \int_a^b g'(x) \mu\left[x, +\infty\right)dx.$$
Thus, we get for $N$ large enough,
$$ \EE\left( e^{\lambda| X_{i,j}|^2 \Car_ {(\log N)^d < |X_{i,j}|< \veps N^{\frac{1}{2} + \frac{1}{p}}}}\right) \leq 1 + \int_{(\log N)^d}^{\veps N^{\frac{1}{2} + \frac{1}{p} }} 2\lambda x e^{f(x)} dx,$$
with $f(x) = \lambda x^2 -\gamma x^{\alpha}$, and $\gamma$ is as in \eqref{tail}. Let 
$$\lambda = \frac{\alpha\gamma}{2} \veps^{\alpha -2} N^{-(2-\alpha)\left( \frac{1}{2} + \frac{1}{p}\right)}.$$
With this choice of $\lambda$, one can easily check that $f$ is non-increasing on $[(\log N)^d, \veps N^{\frac{1}{2} + \frac{1}{p} }]$. Thus,
\begin{align*}
\EE\left( e^{\lambda \left|X_{i,j}\right|^2 \Car_ {(\log N)^d < \left|X_{i,j}\right|< \veps N^{\frac{1}{2} + \frac{1}{p}}}}\right) &\leq 1+2\lambda \veps^2 N^{1+\frac{2}{p}} e^{f\left((\log N)^d\right)}\\
&\leq 1 +\alpha\gamma\veps^{\alpha }N^{\alpha\left(\frac{1}{2}+ \frac{1}{p}\right)}e^{f\left((\log N)^d\right)}.
\end{align*}
But for $N$ large enough,
$$f( (\log N)^d) = \frac{\alpha \gamma}{2} \veps^{\alpha-2} N^{-(2-\alpha)\left( \frac{1}{2} + \frac{1}{p} \right) } (\log N)^{2d} - \gamma (\log N)^{\alpha d}\leq -\frac{\gamma}{2} \left( \log N\right)^{\alpha d}.$$
As $\alpha d >1$, we get for $N$ large enough,
$$ \EE\left( e^{\lambda \left|X_{i,j}\right|^2 \Car_ {(\log N)^d < \left|X_{i,j}\right|< \veps N^{\frac{1}{2} + \frac{1}{p}}}}\right) \leq 1 + e^{-\frac{\gamma}{4} \left( \log N \right)^{\alpha d}}\leq \exp\Big( e^{-\frac{\gamma}{4} \left( \log N\right)^{\alpha d}}\Big).$$
Then,
\begin{equation}\label{probatrB} \PP\left( \tr \left|B^{\veps}\right|^p \geq t N \right) \leq \exp\Big( -\frac{\alpha\gamma}{4}\veps^{\alpha-2} N^{\alpha\left(\frac{1}{2}+\frac{1}{p}\right)}t^{\frac{2}{p}} \Big) \exp\Big( N^2 e^{-\frac{\gamma}{2} ( \log N)^{\alpha d}}\Big).\end{equation}
Since $\alpha d>1$, we get
$$ \lim_{t \to +\infty} \bornesupalpha \PP\left( \tr \left|B^{\veps} \right|^p \geq t N \right) = -\infty.$$
\end{proof}

We now turn to the event $\{\tr_N |C^{\veps}|^p > t \}$. As a consequence of Bennett's inequality, we have the following lemma.
\begin{Lem}\label{tensionexpoC}
For any $\veps>0$, 
$$\lim_{t\to +\infty} \bornexpoalpha \PP\left( \tr_N \left|C^{\veps}\right|^p >t \right) = -\infty.$$
\end{Lem}
To prove this lemma, we will first show that at the exponential scale $C^{\veps}$ has a finite number of non-zero entries. 
\begin{Pro}\label{nb non zero entries C}
For all $\veps>0$,
$$ \lim_{r \to +\infty}\bornesupalpha	 \PP\left(\Card\left\{ (i,j) : C^{\veps}_{i,j} \neq 0 \right\} \geq r \right)= -\infty,$$
where $C^{\veps}$ is as in \eqref{cut}.
\end{Pro}
\begin{proof}Let $\veps>0$.
Note that 
$$  \PP\left(\Card\left\{ (i,j) : C^{\veps}_{i,j} \neq 0 \right\} \geq r\right)\leq \PP\left( \sum_{1\leq i\leq j\leq N} \Car_{|X_{i,j}| \geq \veps N^{\frac{1}{2} + \frac{1}{p}}} \geq \frac{r}{2} \right).$$
Let $p_{i,j} = \PP\left(|X_{i,j}|\geq \veps N^{\frac{1}{2} + \frac{1}{p}}\right)$, for $i,j \in \{1,2\}$. From \eqref{tail}, we have
$$ p_{1,1}\vee p_{1,2} = o\left( \frac{1}{N^2} \right).$$
Therefore, it is enough to show that
$$ \lim_{r \to +\infty}\bornesupalpha \PP\left(\sum_{1\leq i\leq j\leq N} \big(\Car_{|X_{i,j}|\geq \veps N^{\frac{1}{2}+\frac{1}{p}}} - p_{i,j}\big) \geq r \right) = -\infty.$$
By Bennett's inequality (see \cite[Theorem 2.9]{Massart}) we have,
$$\PP\left(\sum_{1\leq i\leq j\leq N} \big(\Car_{|X_{i,j}|\geq \veps N^{\frac{1}{2}+\frac{1}{p}}} - p_{i,j}\big) \geq r \right) \leq \exp\Big( -v h\left(\frac{r}{v}\right)\Big),$$
with $h(x) = (x+1)\log (x+1)-x$, and $v = \sum_{i\leq j} p_{i,j} $. From \eqref{tail}, we have for $N$ large enough, 
$$v \leq N^2 e^{-\gamma \veps^{\alpha} N^{\alpha\left(\frac{1}{2}+\frac{1}{p}\right)}}.$$  As $h(x) \underset{+\infty}{\sim} x \log (x)$, we get for $N$ large enough,
$$\PP\left(\sum_{1\leq i,j\leq N} \big(\Car_{|X_{i,j}|\geq \veps N^{\frac{1}{2}+\frac{1}{p}}} - p_{i,j} \big)\geq r \right) \leq \exp\big( -r\gamma\veps^{\alpha} N^{\alpha\left(\frac{1}{2}+\frac{1}{p}\right)} \big)\exp\Big( r\log\Big(\frac{r}{N^2}\Big)\Big),$$
which gives the claim.
\end{proof}
With this result on the number of non-zero entries of $C^{\veps}$, we will see that the matrix $\frac{1}{N}|C^{\veps}|^p$ has a finite number of non-zero entries of order $1$, and that it yields the exponential estimate claimed in Lemma \ref{tensionexpoC}. 
\begin{proof}[Proof of Lemma \ref{tensionexpoC}]
Using the polar decomposition as in \eqref{polar}, and bounding each coefficient of $C^{\veps}$ by $\veps^{-1} N^{1/p}$, we get,
$$
  \tr \left|C^{\veps}\right|^p  \leq  |\mathcal{I}^{\veps}|^{p}N\veps^{-p},$$
where $|\mathcal{I}^{\veps}|$ denotes the number of non-zero entries in $C^{\veps}$.
Due to Lemma \ref{nb non zero entries C}, we get,
$$ \lim_{t \to +\infty} \bornesupalpha \PP\left(  \tr_N \left|C^{\veps}\right|^p >t \right) = -\infty.$$
\end{proof}
At last, we prove the following exponential tightness for $\tr_N|D^{\veps}|^p$.
\begin{Lem}\label{tensionexpoD}It holds
$$\lim_{\veps \to 0}\limsup_{t \to +\infty}\bornesupalpha \PP\left( \tr_N\left|D^{\veps} \right|^p>t\right)  = -\infty,$$
with $D^{\veps}$ as in \eqref{cut}.
\end{Lem}
\begin{proof}
A union bound gives for $N$ large enough,
\begin{equation} \label{tensionD} \PP\left( D^{\veps}\neq 0 \right) \leq N^2\exp\Big( -\gamma \veps^{-\alpha}N^{\alpha (\frac{1}{2}+\frac{1}{p})}\Big),\end{equation}
with $\gamma$ as in \eqref{tail}.
\end{proof}

From \eqref{decouptr}, lemmas \ref{tensionexpoA}, \ref{tensionexpoB}, and \ref{tensionexpoC}, we get for any $\veps>0$,
\begin{align*}
\limsup_{t\to+\infty}&\bornesupalpha \PP\left( \tr_N \left| X_N\right|^p> t \right) \\
&\leq \limsup_{t\to+\infty}  \bornesupalpha \PP\left( \tr_N|D^{\veps}|^p >t \right).
\end{align*}
Taking the limsup as $\veps$ goes to $0$, we see that Lemma \ref{tensionexpoD} yields the exponential tightness claimed in Proposition \ref{tensionexpo}.
\end{proof}

\subsection{Exponential equivalences}

\subsection{First step}
We will prove in this section that we can ignore in the deviations of $\tr_N X_N^p$ the contributions of the large entries, namely those such that $|X_{i,j}| >  \veps^{-1} N^{\frac{1}{2} + \frac{1}{p}}$, and the contributions of the intermediate entries, that is $(\log N)^d < |X_{i,j}|<\veps N^{\frac{1}{2}+\frac{1}{p}}$. More precisely, we will prove the following exponential approximation.

\begin{Pro}\label{firststep}For any $t>0$, 
$$\lim_{\veps \to 0} \bornesupalpha \PP\left( \left|\tr_N X_N^p - \tr_N\left(A+C^{\veps}\right)^p \right| >t \right) = -\infty,$$
with $A$ and $C^{\veps}$ are as in \eqref{cut}. In other words, $(\tr_N(A+C^{\veps})^p)_{N\in \NN }$ are exponentially good approximations of $(\tr_N X_N^p)_{N\in\NN}$.
\end{Pro}

\begin{proof}Let $\tau>0$. Define the compact subset, 
$$ \mathcal{K}_{\tau} = \left\{ \mu \in \M_1(\RR) : \langle \mu, |x|^p\rangle \leq \tau\right\},$$
 where $\M_1(\RR)$ denotes the set of probability measures on $\RR$. 
As the function which associates to a probability measure $\mu$ on $\RR$, its $p^{\text{th}}$ moment, $\langle \mu,  x^p\rangle$, is continuous for the $p$-Wasserstein distance, we get that restricted to  $\mathcal{K}_{\tau} $, it is uniformly continuous. Applying this uniform continuity to  spectral measures of Hermitian matrices, using the fact that 
$$ W_p(\mu_A, \mu_B) \leq \Big( \tr_N|A-B|^p\Big)^{1/p},$$
for any two Hermitian matrices $A$ and $B$, with spectral measures $\mu_A$, $\mu_B$,
we get that there exists a non-negative function $h$ depending on $\tau$, satisfying $h(t) \to 0$ as $t\to 0$, such that for any $X,Y \in \mathcal{H}^{(\beta)}_N$, if
$$\tr_N |X|^p \leq \tau,  \text{ and }\left|\tr_N X^p  - \tr_N Y^p \right|> t,$$
for some $t>0$, then,
$$\tr_N\left|X-Y\right|^p>h(t).$$
But, from Proposition \ref{tensionexpo}, we know that $(\tr_N |X_N|^p)_{N\in \NN}$  is exponentially tight, therefore, it is enough to show that for any $\tau>0$,
$$\lim_{\veps \to 0}\limsup_{N\to +\infty} N^{-\alpha\left( \frac{1}{2} + \frac{1}{p} \right)}\log \PP\left( \left|\tr_N X_N^p -\tr_N (A+C^{\veps})^p \right| >t, \tr_N|X_N|^p \leq \tau \right) =-\infty.$$
Let $\tau>0$. With the previous observation, we get for any $t>0$,
$$\PP\left(\left|\tr_N X_N^p - \tr_N(A+C^{\veps})^p\right| >t,\tr_N|X_N|^p \leq \tau \right)\leq \PP\Big( \tr_N|B^{\veps}+D^{\veps}|^p > h(t) \Big).$$
By the triangular inequality for the $p$-Schatten norm, we get
\begin{align} \PP\Big(|\tr_N X_N^p -& \tr_N(A+C^{\veps})^p| >t,\tr_N|X_N|^p \leq \tau \Big)\nonumber\\
&\leq \PP\left(\tr_N|B^{\veps}|^p > \frac{h(t)}{2^p} \right)+\PP\left(\tr_N|D^{\veps}|^p > \frac{h(t)}{2^p} \right)\label{BD}.\end{align}
But, on one hand \eqref{probatrB} yields
$$\lim_{\veps\to 0} \bornexpoalpha  \PP\left(\tr_N|B^{\veps}|^p > \frac{h(t)}{2^p} \right) = -\infty,$$
and on the other hand, \eqref{tensionD} gives
$$\lim_{\veps\to 0} \bornexpoalpha \PP\left(\tr_N|D^{\veps}|^p > \frac{h(t)}{2^p} \right)=-\infty.$$
This concludes the proof of Proposition \ref{firststep}, taking the limsup as $N$ goes to $+\infty$ at the exponential scale, and then the limsup as $\veps$ goes to $0$ in \eqref{BD}.

\end{proof}

\subsection{Second step}
We show here that in the study of the deviations of $\tr_N(A+C^{\veps})^p$, we can replace $A$ by a matrix $H$ independent of $X$, and that $\tr_N(H+C^{\veps})^p$ is exponentially equivalent to its conditional expectation given the $\sigma$-algebra $\mathcal{F}$, generated by the $X_{i,j}$ such that $|X_{i,j}|>(\log N)^d$. More precisely, we will prove the following result.

\begin{Pro}\label{concentdecouplage}
Let $\mathcal{F}$ be the $\sigma$-algebra generated by the variables $X_{i,j}\Car_{|X_{i,j}|>\left(\log N\right)^d}$. Let $H$ be a random Hermitian matrix independent of $X$, such that $(H_{i,j})_{i\leq j}$ are independent, and for all $1\leq i\leq N$, $H_{i,i}$ has the same law as $X_{1,1}/\sqrt{N}$ conditioned on $\{|X_{1,1}|\leq (\log N)^d\}$, and for all $i<j$, $H_{i,j}$ has the same law as $X_{1,2}/\sqrt{N}$ conditioned on $\{|X_{1,2} |\leq (\log N)^d\}$. 

For any $t>0$,
$$\bornexpoalpha \PP\left( \left| \tr_N X_N^p - \EE_{\mathcal{F}}\tr_N\left(H+C^{\veps}\right)^p\right|>t \right) = -\infty,$$
where $\EE_{\mathcal{F}}$ denotes the conditional expectation given $\mathcal{F}$.
\end{Pro}
\begin{proof}
By Proposition \ref{firststep}, we know that $(\tr_N(A+C^{\veps})^p)_{N\in \NN, \veps>0}$ are exponentially good approximations of $(\tr_N X_N^p)_{N\in\NN}$, therefore it is enough to show that for all $\veps>0$, and $t>0$,
$$\bornexpoalpha \PP\left( \left|\tr_N\left(A+C^{\veps}\right)^p - \EE_{\mathcal{F}}\tr_N\left(H+C^{\veps}\right)^p \right|>t \right) = -\infty.$$
From Proposition \ref{tensionexpoC}, we see that is actually sufficient to show that for any $r\in \NN$, 
 $$\bornexpoalpha \PP\left( \left|\tr_N\left(A+C^{\veps}\right)^p - \EE_{\mathcal{F}}\tr_N\left(H+C^{\veps}\right)^p \right|>t, |\mathcal{I}_{\veps}|\leq r \right) = -\infty,$$
where
$$ \mathcal{I}_{\veps}=\Big\{ (i,j)\in\{1,...,N\}\times \{1,...,N\} : C^{\veps}_{i,j} \neq 0\Big\}.$$
Note that $C^{\veps}$ is $\mathcal{F}$-measurable, and given $\mathcal{F}$, $A$ has independent up-diagonal entries bounded by $(\log N)^d/\sqrt{N}$. Moreover, using the triangle inequality for the $2(p-1)$-Schatten norm, we get
$$\tr (\EE A +C^{\veps})^{2(p-1)} \leq 2^{2(p-1)}\max \Big( \tr( \EE A )^{2(p-1)}, \tr( C^{\veps})^{2(p-1)}\Big).$$
On one hand, we have, expanding the trace and bounding each entry of $C^{\veps}$ by $\veps^{-1}N^{1/p}$,
$$\tr(C^{\veps})^{2(p-1)} \leq |\mathcal{I}_{\veps}|^{2(p-1)} \veps^{-2(p-1)} N^{2-\frac{2}{p}},$$
and on the other hand we have from \eqref{momentA} that $\tr(\EE A)^{2(p-1)} =o(1)$.
Therefore, we can apply the result of Proposition \ref{conctr} for the trace of $(A+C^{\veps})^p$ under the conditional probability given $\mathcal{F}$. As $\alpha<2$, we get that for any $t>0$, and $r\in \NN$,
$$\bornexpoalpha \PP\left( \left|\tr_N\left(A+C^{\veps}\right)^p - \EE_{\mathcal{F}}\tr_N\left(A+C^{\veps}\right)^p\right|>t, |\mathcal{I}_{\veps}|\leq r \right) = -\infty.$$
We will use the same decoupling argument as in \cite{Bordenave}, to remove the dependency between $A$ and $C^{\veps}$.
Let $I = \left\{ (i,j) : |X_{i,j}|> (\log N)^d \right\}$. Define $A'$ the $N\times N$ matrix with $(i,j)$-entry
\begin{equation}\label{defA'} A'_{i,j} = A_{i,j}\Car_{(i,j)\notin I} + H_{i,j}\Car_{(i,j)\in I}.\end{equation}
Note that $A'$ and $H$ are both independent of $\mathcal{F}$ and have the same law. Therefore, 
$$\EE_{\mathcal{F}}\tr_N\left(A' +C^{\veps}\right)^p = \EE_{\mathcal{F}}\tr_N\left(H +C^{\veps}\right)^p.$$
Due to the triangular inequality and Lemma \ref{tensionexpoC}, it only remains to prove that for any $t>0$, and any $\tau>0$,
 $$\bornexpoalpha \PP\left( \left|\EE_{\mathcal{F}}\tr_N\left(A+C^{\veps}\right)^p - \EE_{\mathcal{F}}\tr_N\left(A'+C^{\veps}\right)^p \right|>t, \tr_N|C^{\veps}|^p \leq \tau \right) = -\infty.$$
But, using again the triangular inequality for the $p$-Schatten norm, we get
$$\EE_{\mathcal{F}} \tr_N|A'+C^{\veps}|^p \leq 2^p\max\left( \EE \tr_N|H|^p, \tr_N|C^{\veps}|^p \right).$$
With the same argument as in the proof of Lemma \ref{tensionexpoA} we have
$$ \EE \tr_N|H|^p \underset{N\to +\infty}{\longrightarrow } \langle \sigma_{sc}, |x|^p \rangle.$$
Arguing as in the proof of Lemma \ref{tensionexpoA}, we see that it is sufficient to show that for any $t>0$,
$$\bornexpoalpha \PP\left( W_p(\EE_{\mathcal{F}}\mu_{A+C^{\veps}}, \EE_{\mathcal{F}} \mu_{A'+C^{\veps}})>t \right) = -\infty,$$
where $\mu_{A+C^{\veps}}$ and $\mu_{A'+C^{\veps}}$ denote the spectral measures of $A+C^{\veps}$ and $A'+C^{\veps}$. But,
$$ W_p(\EE_{\mathcal{F}}\mu_{A+C^{\veps}}, \EE_{\mathcal{F}} \mu_{A'+C^{\veps}})^p \leq \EE_{\mathcal{F}}\tr_N|A-A'|^p,$$
and besides, expanding the trace using the polar decomposition, we get
\begin{equation}
\EE_{\mathcal{F}}\tr_N\left|A-A'\right|^p \leq  c_0\frac{|I|^{p}}{N^{1+p/2}},\label{estimation distance couplage}
\end{equation}
where $c_0$ is constant independent of $N$ such that,
 $$\max\left( \EE \left|\sqrt{N}H_{1,1}\right|^p, \EE \left|\sqrt{N}H_{1,2}\right|^p\right)  \leq c_0.$$
Thus, in order to control $\EE_{\mathcal{F}}\tr_N\left|A-A'\right|^p$, we need to make sure that $I$ contains no more than $t N^{1+p/2}$ indices, for any $t>0$, at the exponential scale $N^{\alpha(\frac{1}{2}+\frac{1}{p})}$. By a argument similar as in the proof of Proposition \ref{nb non zero entries C}, we get  the following lemma.

\begin{Lem}\label{control nombre zero dans A}Let $I = \left\{ (i,j) : |X_{i,j}|> (\log N)^d\right\}$.
For $\delta>0$, we define the event, $$F_{\delta} = \left\{ |I| \leq \frac{\delta}{c_0} N^{1+2/p}  \right\}.$$
It holds that
$$\bornexpoalpha \PP\left( F_{\delta}^c \right)= -\infty.$$
\end{Lem}
Using \eqref{estimation distance couplage}, and Lemma \eqref{control nombre zero dans A}, we get the claim.

\end{proof}
\subsection{Third step}\label{third step}
We showed in Proposition \ref{concentdecouplage} that $(\EE_{\mathcal{F}}\tr_N\left(H+C^{\veps}\right)^p)_{N\in \NN, \veps>0}$ are exponentially good approximations of $(\tr_N X_N^p)_{N\in \NN}$. We will prove now that we can approximate $\EE_{\mathcal{F}}\tr_N\left(H+C^{\veps}\right)^p$ at the exponential scale $N^{\alpha\left(\frac{1}{2} + \frac{1}{p}\right)}$, by $\EE \tr_N H^p + \tr_N (C^{\veps})^p$, and then by $\langle\sigma_{sc}, x^p \rangle + \tr_N (C^{\veps})^p$. This will give good exponential approximations of $(\tr_N X_N^p)_{N\in \NN}$, as stated in the following proposition.

\begin{Pro}\label{equivtr}
For any $t>0$,
$$\lim_{\veps \to 0}\bornesupalpha \PP\left(\left| \tr_NX_N^p - \left\langle \sigma_{sc} , x^p \right\rangle - \tr_N\left(C^{\veps}\right)^p \right|>t \right) = -\infty,$$
where $A$ and $C^{\veps}$ are as in \eqref{cut}.
\end{Pro}

In order to prove that $\EE \tr_N H^p + \tr_N (C^{\veps})^p$ is an exponential equivalent of $\EE_{\mathcal{F}}\tr_N\left(H+C^{\veps}\right)^p$, we will need the following deterministic lemma.

\begin{Lem}
Let $p\geq 2$. Let $H$ and $C$ be two Hermitian matrices of size $N$. Assume that $C$ is of rank at most $r$. We have
$$ \left|\tr\left(H + C\right)^p - \tr H^p - \tr C^p \right| \leq 2^p r \max_{1\leq k\leq p-1} ||H||^k ||C||^{p-k},$$
where $|| \ ||$ denotes the operator norm.
\end{Lem}

\begin{proof}Expanding the sum we get
$$ \tr\left(H+C\right)^p =\sum_{k=0}^p \sum_{\underset{| \{i : M^{(i)} = H\}| =k}{M^{(i)} \in \{H, C\}}} \tr\left(M^{(1)}...M^{(p)}\right).$$
Let $k\in \{1,...,p-1\}$, and let $M^{(1)}$,...,$M^{(p)}$ be matrices such that $M^{(i)} \in \{H, C\}$, and $\Card \{ i : M^{(i)} = H\} =k$ . Let $(\eta_j)_{1\leq j\leq N}$ be an orthonormal basis of eigenvectors for $C$ such that $\eta_{r+1},...,\eta_{N}$ are in the kernel of $C$. Using the cyclicity of the trace, we can assume $M^{(p)} = C$. Assuming $M^{(p)} = C$, we get
 \begin{align*}
\left|\tr\left(M^{(1)}...M^{(p)}\right)\right|& = \left|\sum_{j=1}^N \left\langle M^{(1)}...M^{(p)} \eta_j, \eta_j \right \rangle\right| \\
&= \left|\sum_{j=1}^r \left\langle M^{(1)}...M^{(p)} \eta_j, \eta_j \right \rangle \right| \\
& \leq r || H||^k ||C||^{p-k},
\end{align*}
which ends the proof of the claim.

\end{proof}

%
%

%

\begin{proof}
Note that the same argument as in the proof of Lemma \ref{tensionexpoA} yields
$$ \EE \tr_N H^p \underset{N\to +\infty}{\longrightarrow} \left\langle \sigma_{sc}, x^p \right\rangle,$$
Therefore, due to Proposition \ref{concentdecouplage},
we only need to prove that for any $\veps>0$,
\begin{equation*} \label{decompotr} \bornexpoalpha\PP\left(\left|\EE_{\mathcal{F}}\tr_N\left(H+C^{\veps}\right)^p - \EE \tr_N H^p - \tr_N \left(C^{\veps}\right)^p \right| >t \right) = -\infty.\end{equation*}
Using Lemma \ref{decompotrace} and the fact that the rank of a matrix is bounded by the number of its non-zero entries, we have
$$\left|\EE_{\mathcal{F}}\tr_N\left(H+C^{\veps}\right)^p - \EE \tr_NH^p- \tr_N \left(C^{\veps}\right)^p \right|  \leq \frac{2^p}{N} |\mathcal{I}_{\veps}|\max_{1\leq k\leq p-1}\left\{ ||C^{\veps}||^{p-k}\EE||H||^{k}\right\},$$
where $\mathcal{I}_{\veps}$ denotes the set of indices $(i,j)$ such that $C^{\veps}_{i,j} \neq 0$.
But,
$$ || C^{\veps} || \leq \left|\mathcal{I}_{\veps}\right| \sup_{i,j} \left|C_{i,j} \right| \leq  \left|\mathcal{I}_{\veps}\right|\veps^{-1} N^{1/p}.$$
Thus,
$$\left|\EE_{\mathcal{F}}\tr_N\left(H+C^{\veps}\right)^p - \EE \tr_NH^p- \tr_N \left(C^{\veps}\right)^p \right|  \leq \frac{2^p\veps^{-p+1}}{N^{1/p}} |\mathcal{I}_{\veps}|^p\max_{1\leq k\leq p-1}\EE||H||^{k}.$$
But we know from \cite[Theorem 2.1.22, Exercice 2.1.27]{Guionnet}  that $||X||$ converges in all $L^p$ spaces to $2$, and we have
$$ \EE ||X-H||^p= \EE ||X-A'||^p\leq \EE \tr|X-A'|^p,$$
where $A'$ is as in \eqref{defA'}. With the same argument as in Lemma \ref{tensionexpoA}, we get 
$$ \EE \tr|X-A'|^p = o(1).$$
Thus, for any $k\in \{1,...,p\}$, $\EE||H||^k$ is bounded.
We can find a constant $M_p>0$ such that, 
 $$\left|\EE_{\mathcal{F}}\tr_N\left(A+C^{\veps}\right)^p - \EE \tr_N A^p - \tr_N \left(C^{\veps}\right)^p \right|  \leq  M_p |\mathcal{I}_{\veps}|^p N^{-\frac{1}{p}}.$$
Thus, for any $t>0$, and $r\in \NN$,
$$ \bornexpoalpha\PP\left(\left|\EE_{\mathcal{F}}\tr_N\left(A+C^{\veps}\right)^p - \EE \tr_N A^p - \tr_N \left(C^{\veps}\right)^p \right| >t, |\mathcal{I}_{\veps}|\leq r \right) = -\infty.$$
Invoking Lemma \ref{nb non zero entries C}, we get the claim.
\end{proof}

\subsection{A large deviations principle for $\tr_N X_N^p$}\label{LDP tr}
We proved in the previous section that $(\langle \sigma_{sc}, x^p \rangle + \tr_N(C^{\veps})^p)_{\veps>0, N \in \NN}$ are exponentially good approximations of $(\tr_N X_N^p)_{N\in \NN}$ at the exponential scale considered. The aim of this section is to show that we can derive a LDP for each $\veps>0$ for $(\tr_N(C^{\veps})^p)_{N\in \NN}$, using the contraction principle, and deduce a LDP for $(\tr_N X_N^p)_{N\in \NN}$. 


In the view of applying a contraction principle for the sequence $(\tr_N(C^{\veps})^p)_{N\in \NN}$, we need to find a good space to embed $C^{\veps}$ so that we can define a trace which will be continuous.
For every $r\in \NN$, we define
$$ \mathcal{E}_r = \{A\in \cup_{n\geq 1}\mathcal{H}^{(\beta)}_n : \mathrm{Card}\{(i,j) : A_{i,j} \neq 0\} \leq r \}.$$
For any $n\in \NN$, let $\mathcal{S}_n$ be the symmetric group on the set $\{1,...,n\}$. Let $\mathcal{S}$ denote the group $\cup_{n\in \NN}\mathcal{S}_n$. We denote $\widetilde{\mathcal{E}}_r$ the set of equivalence classes of $\mathcal{E}_r$ under the action of $\mathcal{S}$, which is defined by
$$ \forall \sigma \in \mathcal{S}, \forall A \in  \mathcal{E}_r, \ \sigma.A = M_{\sigma}^{-1}AM_{\sigma} = \left(A_{\sigma(i), \sigma(j)}\right)_{i,j},$$
where $M_{\sigma}$ denote the permutation matrix associated with the permutation $\sigma$ i.e $M_{\sigma} =(\delta_{i,\sigma(j)})_{i,j}$.

Let $\mathcal{H}^{(\beta)}_r/\mathcal{S}_r$ be the set of equivalence classes of $\mathcal{H}^{(\beta)}_r$ under the action of the symmetric group $\mathcal{S}_r$. Note that any equivalence class of the action of $\mathcal{S}$ on $\mathcal{E}_r$ has a representative in $\mathcal{H}^{(\beta)}_r$. This defines an injective map from $\widetilde{\mathcal{E}}_r$ into $\mathcal{H}^{(\beta)}_r/  \mathcal{S}_r$. Identifying $\widetilde{\mathcal{E}}_r$ to a subset of $\mathcal{H}^{(\beta)}_r / \mathcal{S}_r$, we equip $\widetilde{\mathcal{E}}_r$ of the quotient topology of $\mathcal{H}^{(\beta)}_r / \mathcal{S}_r$.
This topology is metrizable by the distance $\tilde{d}$ given by
\begin{equation} \label{definition distance topo quotient}\forall \tilde{A}, \tilde{B} \in \widetilde{\mathcal{E}}_r, \ \tilde{d}\left(\tilde{A}, \tilde{B}\right) = \min_{\sigma, \sigma' \in \mathcal{S}}\max_{i,j} \left|B _{\sigma(i), \sigma(j)} - A_{\sigma'(i),\sigma'(j)}\right|,\end{equation}
where $A$ and $B$ are two representatives of $\tilde{A}$ and $\tilde{B}$ respectively.

Since the trace is continuous and  invariant by conjugation, we can define the trace on $\mathcal{H}^{(\beta)}_r/ \mathcal{S}_r$ and it will be still continuous. Therefore, the trace on  $\widetilde{\mathcal{E}}_r$ is continuous for the topology we defined above.

Let $\veps>0$. Let $\PP^{\veps}_{N,r}$ denote the law of $C^{\veps}/N^{1/p}$ conditioned on the event $\{C^{\veps} \in \mathcal{E}_r\}$, and $\widetilde{\PP}^{\veps}_{N,r}$ the push-forward of $\PP^{\veps}_{N,r}$ by the projection $\pi : \mathcal{E}_r \to \widetilde{\mathcal{E} }_r$.
With these preliminary definitions, we can now state the LDP result for $(\widetilde{\PP}^{\veps}_{N,r})_{N\in \NN}$. The result is almost identical as \cite[Proposition 7.1]{LDPevNG}, the only difference being the choice of truncation of the entries. Thus, the rate function is identical, and only the speed is different. We refer the reader to  \cite{LDPevNG} for the proof of the following proposition.
\begin{Pro} \label{LDP nb entree fixe}
Let $r\in \NN$ and $\veps>0$. Then $(\widetilde{\PP}^{\veps}_{N,r})_{N\in \NN}$ satisfies a large deviations principle with speed $N^{\alpha\left(\frac{1}{2}+\frac{1}{p}\right)}$, and good rate function $I_{\veps, r}$ defined  for all $\tilde{A}\in \widetilde{\mathcal{E}}_r$ by
\begin{equation} I_{\veps, r}\left(\tilde{A}\right) = 
\begin{cases}
b\sum_{i\geq 1}\left|A_{i,i}\right|^{\alpha} + \frac{a}{2}\sum_{i\neq j} \left|A_{i,j}\right|^{\alpha} & \text{ if } A \in \mathcal{D}_{\veps,r },\\
+\infty & \text{ otherwise,}
\end{cases}
\label{fntaux}\end{equation} 
where $A$ is a representative of the equivalence class $\tilde{A}$ and 
$$\mathcal{D}_{\veps,r} =\left \{ A \in \mathcal{E}_r : \forall i\leq j, \ A_{i,j} = 0 \text{ or } \ \veps\leq \left|A_{i,j}\right| \leq \veps^{-1}, \text{ and } A_{i,j}/|A_{i,j}|  \in \mathrm{supp}(\nu_{i,j}) \right\},$$ 
with $\nu_{i,j} = \nu_1$ if $i=j$, and $\nu_{i,j} = \nu_{2}$ if $i<j$, where $\nu_1$ and $\nu_2$ are defined in definition \ref{WG}.
\end{Pro}

We are now ready to use a contraction principle to prove that $(\tr_N(C^{\veps})^p)_{N\in\NN}$ follows a LDP for any $\veps>0$. The use of the contraction principle is made possible by the fact that the push-forward of $\tilde{\PP}^{\veps}_{N,r}$ by the map $A\mapsto \tr A^p$ on $\cup_{n\in\NN} H_n(\CC)$, are exponentially good approximations of $(\tr_N(C^{\veps})^p)_{N\in \NN}$.

\begin{Pro}\label{LDPtrC}
Let  $\veps>0$. The sequence $(\tr_N\left(C^{\veps}\right)^p)_{N\in \NN}$ satisfies a large deviations principle of speed $N^{\alpha\left( \frac{1}{2}+ \frac{1}{p} \right)}$, and good rate function $J_{\veps}$ defined for all $x \in \RR$ by,
$$J_{\veps}(x) = \inf \left\{ I_{\veps}\left(A\right) :   x= \tr A^p, A\in \cup_{n\in \NN} \mathcal{H}^{(\beta)}_n \right\},$$
where 
\begin{equation}\forall A \in \cup_{n\in \NN} \mathcal{H}^{(\beta)}_n, \  I_{\veps}\left(A\right) = 
\begin{cases}
b\sum_{i\geq 1}\left|A_{i,i}\right|^{\alpha} + \frac{a}{2}\sum_{i\neq j} \left|A_{i,j}\right|^{\alpha} & \text{ if } A \in \mathcal{D}_{\veps },\\
+\infty & \text{ otherwise,}
\end{cases}
\label{fntaux}\end{equation} 
where $\mathcal{D}_{\veps} =\cup_{r \in \NN}\mathcal{D}_{\veps, r}$, with  $\mathcal{D}_{\veps, r}$ as in Proposition \ref{LDP nb entree fixe}.
\end{Pro}

\begin{proof}
Let $r\in \NN$. We denote by $f$ the function $\tilde{A} \in \tilde{\mathcal{E}_r} \mapsto \tr A^p$, with $A$ a representative of $\tilde{A}$. As the trace is invariant by conjugation, $f$ is well defined.  We define the push-forward of $\tilde{P}^{\veps}_{N,r}$ by the map $f$, $$\nu_{N,r} = \tilde{P}^{\veps}_{N,r}\circ f^{-1}.$$
Note that $\nu_{N,r}$ is the law of $\tr_N(C^{\veps})^p$ conditioned on the event $\{C^{\veps} \in \mathcal{E}_r\}$. We will show that $(\nu_{N,r})_{N, r\in \NN}$ are exponentially good approximations of $(\tr_N(C^{\veps})^p)_{N\in \NN}$. Let $Y_{N,r}$ be random variable independent of $C^{\veps}$, and distributed according to $\nu_{N,r}$. Let
$$ Z_{N,r} = \tr_N\left(C^{\veps}\right)^p\Car_{C^{\veps} \in \mathcal{E}_r} + Y_{N,r}\Car_{C^{\veps} \notin \mathcal{E}_r}.$$
Thus, $Z_{N,r}$ and $Y_{N,r}$ have the same law $\nu_{N,r}$. Furthermore, for any $t>0$,
$$ \PP\left( \left|Z_{N,r} - \tr_N\left(C^{\veps}\right)^p \right|> t \right) \leq \PP\left( C^{\veps} \notin \mathcal{E}_r \right).$$
By Proposition \ref{nb non zero entries C}, we get
$$\bornexpoalpha  \PP\left( \left|Z_{N,r} - \tr_N\left(C^{\veps}\right)^p \right|> t \right) = -\infty,$$
which shows that $(\nu_{N,r})_{N,r \in \NN}$ are exponentially good approximations of $(\tr_N(C^{\veps})^p)_{N\in \NN}$. 

For each $r\in \NN$, the function $f$ restricted to $\widetilde{\mathcal{E}}_r$ is continuous for the topology we equipped $\widetilde{\mathcal{E}}_r$. Note that as $C^{\veps}$ has entries bounded by $\veps^{-1}N^{1/p}$, $\nu_{N,r}$ is compactly supported uniformly in $N$. Thus, $(\nu_{N,r})_{N\geq1}$ is exponentially tight,  the contraction principle (see \cite{Zeitouni}[Theorem 4.2.1]) yields that $(\nu_{N,r})_{N \in \NN}$ follows a LDP principle with speed $N^{\alpha ( \frac{1}{2} + \frac{1}{p})}$ and good rate function $J_{\veps, r}$ given by
$$ J_{\veps, r}(x) = \inf\left\{ I_{ \veps, r}(\tilde{A}) : \tilde{A}\in \widetilde{\mathcal{E}}_r,  x = f(\tilde{A}) \right\},$$
where $I_{\veps,r}$ is defined in Proposition \ref{LDP nb entree fixe}. We can re-write this rate function  as
$$ J_{\veps, r}(x) =
 \inf\left\{ I_{ \veps}(A) : A\in \mathcal{E}_r, \ x =f(A) \right\},$$
where $f$ denote as well the function $A\mapsto \tr(A)^p$ on $\cup_{n\in \NN} \mathcal{H}^{(\beta)}_n$, and where $I_{\veps}$ is defined in \eqref{fntaux}.
By \cite[Theorem 4.2.16]{Zeitouni}, we deduce that $(\tr_N(C^{\veps})^p)_{N\in \NN}$ satisfies  a weak LDP with speed $N^{\alpha(\frac{1}{2}+\frac{1}{p})}$, and rate function $J_{\veps}$ defined by
$$\forall x \in \RR, \ J_{\veps}(x) = \sup_{\delta>0}\liminf_{r\to +\infty} \inf_{|y-x|<\delta} J_{\veps,r}(y).$$
As $J_{\veps,r}$ is non-increasing in $r$, we have
$$J_{\veps}(x) = \sup_{\delta>0}\inf_{r\in \NN} \inf_{|y-x|<\delta} J_{\veps,r}(y) =  \sup_{\delta>0} \inf_{|y-x|<\delta}  \inf_{r\in \NN} J_{\veps,r}(y).$$
Let $\Phi$ be the function defined by
$$\forall x \in \RR, \ \Phi(x) =  \inf_{r\in \NN} J_{\veps,r}(x).$$
Thus,
$$J_{\veps}(x) = \sup_{\delta>0} \inf_{|y-x|<\delta} \Phi(y).$$
We see that it suffices to show that $\Phi$ is lower semi-continuous to conclude that $J_ {\veps} = \Phi$. We will prove in fact that $ \Phi$ has compact level sets. 

Let $\tau>0$. Let $x\in \RR$, such that $\Phi(x) \leq \tau$. Then
$$\Phi(x) = \left\{ I_{\veps}(A) : x =f(A),  I_{\veps}(A) \leq 2\tau\right\}.$$
But for any $A\in \cup_{n\in \NN} H_n(\CC)$ such that $I_{\veps}(A)< +\infty$, we have
$$(b \wedge \frac{a}{2} )\veps^{\alpha} \Card\left\{ (i,j) : A_{i,j} \neq 0\right\} \leq I_{\veps}(A).$$
Thus taking $r$ such that $(b \wedge \frac{a}{2})\veps^{\alpha}\leq \tau$, we get 
\begin{align*}
\Phi(x) &=  \left\{ I_{\veps}(A) : x =f(A),  I_{\veps}(A) \leq 2\tau, A \in \mathcal{E}_r\right\}\\
& =  \left\{ I_{\veps,r}(\tilde{A}) : x =f(A), \tilde{A} \in \widetilde{\mathcal{E}}_r\right\}.
\end{align*}
Since $f$ is continuous on $\widetilde{\mathcal{E}}_r$ and $I_{\veps,r}$ is a good rate function, we have
$$ \left\{ x \in \RR : \Phi(x) \leq \tau \right\} = \left\{ f(\tilde{A}) : I_{\veps, r} (\tilde{A}) \leq \tau, \tilde{A} \in \widetilde{\mathcal{E}}_r \right\}.$$
As $f$ is continuous on $\mathcal{E}_r$, and $I_{\veps,r}$ is a good rate function, we deduce that the $\tau$-level sets of $\Phi$ are compact. Therefore $J_{\veps} = \Phi$.

\end{proof}
We are now ready to conclude the proof of Theorem \ref{LDPtraceWG}

\begin{proof}[Proof of Theorem \ref{LDPtraceWG} ] 
By Proposition \ref{equivtr}, $(\langle \sigma_{sc} , x^p \rangle + \tr_N(C^{\veps})^p)_{N\in \NN, \veps>0}$ are exponentially good approximations of $(\tr_N X_N^p)_{N\in \NN }$. 
We deduce from  Proposition \ref{LDPtrC} that for each $\veps>0$,  the sequence $(\langle \sigma_{sc} , x^p \rangle + \tr_N(C^{\veps})^p)_{N\in \NN }$ satisfies a LDP with speed $N^{\alpha( \frac{1}{2}+\frac{1}{p})},$ and with good rate function $\psi_{\veps}$ defined by
$$\psi_{\veps}(x) = \begin{cases}
J_{\veps}\left(x-C_{p/2}\right) & \text{ if $p$ is even,}\\
J_{\veps}(x) & \text{ if $p$ is odd,}
\end{cases}$$
where $J_{\veps}$ is as in Proposition \ref{LDPtrC}.
Since $(\tr_N X_N^p)_{N\geq 1}$ is exponentially tight by Proposition \ref{tensionexpo}, we deduce from \cite[Theorem 4.2.16]{Zeitouni} that $(\tr_N X_N^p)_{N\in \NN}$ satisfies a LDP with speed $N^{\alpha( \frac{1}{2} + \frac{1}{p})}$ and rate function $J_p$ defined by
$$ \forall x \in \RR,  \ J_p(x) = \sup_{\delta>0} \limsup_{\veps \to 0} \inf_{|y-x|<\delta}\psi_{\veps}(y).$$
Observe that for any $A\in \cup_{n\in \NN} \mathcal{H}^{(\beta)}_n$, $I_{\veps}(A)$ is non-decreasing in $\veps$. Therefore, $\psi_{\veps}$ is non-decreasing in $\veps$. Thus,
\begin{equation} \label{zeitounifntaux} \forall x \in \RR, \ J_p(x) = \sup_{\delta>0} \inf_{\veps >0} \inf_{|y-x|<\delta}\psi_{\veps}(y).\end{equation}
 Let 
$$ \forall x\in \RR, \ \Phi_p(x) = \begin{cases}
\phi_p\left(x-C_{p/2}\right) & \text{ if $p$ is even,}\\
\phi_p(x) & \text{ if $p$ is odd,}
\end{cases}$$
with
$$ \phi_p(x) = \inf\left\{ I(A) :  x= \tr A^p , A\in \mathcal{D} \right\},$$
where $I$ is defined for any $A\in \cup_{n\geq 1} \mathcal{H}^{(\beta)}_n$, by
$$ I\left( A\right) = b \sum_{i=1}^{+\infty} \left|A_{i,i} \right|^{\alpha} + a \sum_{i<j} \left| A_{i,j} \right|^{\alpha},$$
and $\mathcal{D} = \{ \cup_{n\in  \NN} \mathcal{H}^{(\beta)}_n  : \forall i\leq j, \ A_{i,j} = 0 \text{ or } \  \text{ and } A_{i,j}/|A_{i,j}|  \in \mathrm{supp}(\nu_{i,j})\}$. With these notations we have,
\begin{equation}\label{eqfntaux} J_p(x) =  \sup_{\delta>0} \inf_{|x-y|<\delta} \Phi_p(y).\end{equation}
 As for any $t>0$, and $A\in \cup_{n\in \NN} \mathcal{H}^{(\beta)}_n$, $I(tA) = t^{\alpha} I(A)$, and $\tr\left( tA\right)^p = t^p \tr A^p$, we have for $p$ even,
$$\forall y\in \RR, \  \phi_p(y) = \begin{cases}
\phi_p(1) y^{\alpha/p} & \text{ if } y\geq 0,\\
+\infty & \text{ otherwise,}
\end{cases}$$
and for $p$ odd
$$\forall y\in \RR, \  \phi_p(y) =
\phi_p(1) |y|^{\alpha/p}.$$
Therefore,
$$ \forall x\in \RR, \ \Phi_p(x) = \begin{cases}
\phi_p(1) \left(x-C_{p/2}\right)^{\alpha/p} & \text{ if $p$ is even,}\\
+\infty & \text{ otherwise,}
\end{cases}$$
and if $p$ is odd
$$ \forall x\in \RR, \ \Phi_p(x) =
\phi_p(1)|x|^{\alpha/p}.$$
This shows in particular that $\Phi_p$ is lower semi-continuous. From \eqref{eqfntaux}, we get finally $J_p = \Phi_p$. 

\end{proof}


%
%
%
%
%

\subsection{Computation of $J_p(1)$}\label{CompJ1}
We show here that we can compute the constant $c_p$ appearing in Theorem \ref{LDPtraceWG} when $\alpha \in (0,1]$ and $p$ is even, and we give a lower bound and upper bound in the case where $\alpha \in (1,2)$ and $p$ is even.
\begin{The}With the notations of Theorem \ref{LDPtraceWG}, we have the following :\\
(a). If $p$ is even, 
$$\min\left( b, \frac{a}{2} \right) \leq  c_p \leq \min \left( b, 2^{-\alpha/p}a\right) .$$
(b). If $\alpha \in (0,1]$ and $p$ is even,
$$ c_p = \min \left( b, 2^{-\alpha/p} a \right).$$

\end{The}

\begin{proof} From the proof of Theorem \ref{LDPtraceWG}, we know that 
\begin{equation} \label{constoptim} c_ p = \inf\left\{ I(A) :  1= \tr A^p , A\in \mathcal{D} \right\},\end{equation}
where $I$ is defined for any $A\in \cup_{n\geq 1} \mathcal{H}^{(\beta)}_n$, by $$ I\left( A\right) = b \sum_{i=1}^{+\infty} \left|A_{i,i} \right|^{\alpha} + a \sum_{i<j} \left| A_{i,j} \right|^{\alpha},$$
and $\mathcal{D} = \{ \cup_{n\in  \NN} \mathcal{H}^{(\beta)}_n : \forall i\leq j, \ A_{i,j} = 0 \text{ or } \  A_{i,j}/|A_{i,j}|  \in \mathrm{supp}(\nu_{i,j})\},$
with $\nu_{i,j} = \nu_1$ if $i=j$, and $\nu_{i,j} = \nu_{2}$ if $i<j$, where $\nu_1$ and $\nu_2$ are defined in definition \ref{WG}.

Note that 
$$ c_p \leq \min\left( I(s), I\left(\begin{array}{cc}
0 & 2^{-1/p}e^{i\theta}\\
2^{-1/p}e^{-i\theta} & 0
\end{array}
\right) \right),$$
where $s\in \supp(\nu_1)$, and $\theta \in \supp(\nu_2)$. Thus,
$$c_p \leq \min\left( b, 2^{-\alpha/p} a \right),$$
which proves the upper bound in cases (a) and (b).

On the other hand, we have
\begin{align*}
c_p &\geq \inf\left\{ b \sum_{i=1}^{+\infty} |A_{i,i}|^{\alpha} +  \frac{a}{2} \sum_{i\neq j} |A_{i,j}|^{\alpha} : A\in \cup_{n\in \NN} \mathcal{H}^{(\beta)}_n, 1 =\tr A^p \right\}\\
& \geq \min\left(b, \frac{a}{2}\right) \inf\left\{\sum_{i,j}\left|A_{i,j}\right|^{\alpha} : A\in \cup_{n\in \NN} \mathcal{H}^{(\beta)}_n : \tr A^p =1 \right\}.
\end{align*}
Since $\alpha \in (0,2)$, we know from \cite[Theorem 3.32]{Zhan} that for any $A\in \mathcal{H}^{(\beta)}_n$,
\begin{equation} \label{eqZhan} \sum_{i,j}\left|A_{i,j}\right|^{\alpha}  \geq \sum_{i=1}^n |\lambda_i|^{\alpha},\end{equation}
where $\lambda_1,...,\lambda_n$ are the eigenvalues of $A$. As $\alpha/p\leq 1$, we have
$$\sum_{i=1}^n |\lambda_i|^{\alpha} \geq \Big( \sum_{i=1}^n |\lambda_i|^p \Big)^{\alpha/p} = \Big(\tr|A|^p \Big)^{\alpha/p}\geq \Big|  \tr A^p \Big|^{\alpha/p}.$$
Thus, if $\tr A^p = 1$, we have
$$ \sum_{i,j}\left|A_{i,j}\right|^{\alpha}  \geq 1.$$
We can deduce that 
$$c_p  \geq \min \Big( b,\frac{a}{2} \Big),$$
which proves the lower bound of case (b).\\

Assume now $\alpha \in (0,1)$ and $p$ is even. If $A\in \mathcal{H}^{(\beta)}_n$ is such that $\tr A^p =1$, then
$$ \sup_{\tr |B|^q =1} \tr AB = 1,$$
with $q\geq 1$ such that $\frac{1}{p} + \frac{1}{q} =1$. Thus, we can deduce that 
$$ \forall i \in \{1,...,n\}, |A_{i,i}|\leq 1, \ \forall i,j\in \{1,...,n\}, i \neq j, |A_{i,j}|\leq 2^{-1/p}.$$
Then,
\begin{align}
c_p &\geq \inf\Big\{ b \sum_{i=1}^{+\infty} |A_{i,i}|^{\alpha} +  \frac{a}{2} \sum_{i\neq j} |A_{i,j}|^{\alpha} : A\in \cup_{n\in \NN} \mathcal{H}^{(\beta)}_n, 1 =\tr A^p \Big\}\label{minocp}\\
& \geq \min\Big( b,  2^{-\frac{\alpha}{p}}a\Big) \inf\Big\{  \sum_{i=1}^{+\infty} |A_{i,i}|^{\alpha} + \frac{1}{2} \sum_{i\neq j} |2^{\frac{1}{p}}A_{i,j}|^{\alpha} : A\in \cup_{n\in \NN} \mathcal{H}^{(\beta)}_n, 1 =\tr A^p \Big\}\nonumber\\
& \geq  \min\left( b, 2^{-\frac{\alpha}{p}}a\right) \inf\Big\{  \sum_{i=1}^{+\infty} |A_{i,i}| + \frac{1}{2} \sum_{i\neq j}| 2^{\frac{1}{p}}A_{i,j}| : A\in \cup_{n\in \NN} \mathcal{H}^{(\beta)}_n, 1 =\tr A^p \Big\},\nonumber
\end{align}
where we used in the last inequality the fact the $|A_{i,i}|\leq 1$, and $|A_{i,j}| \leq 2^{-1/p}$ for any $i\neq j$. Thus,
$$
c_p \geq  \min\left( b, 2^{-\frac{\alpha}{p}}a\right) \inf\Big\{  \left( 1-2^{\frac{1}{p}-1} \right)\sum_{i=1}^{+\infty} |A_{i,i}| + 2^{\frac{1}{p} -1} \sum_{i, j} |A_{i,j}| : A\in \cup_{n\in \NN} \mathcal{H}^{(\beta)}_n, 1 =\tr A^p \Big\}.
$$
Using again \cite[Theorem 3.36]{Zhan}, and the triangular inequality, we get
$$c_p \geq  \min\left( b, 2^{-\frac{\alpha}{p}}a\right) \inf_{n\geq 1} \inf\Big\{  \left( 1-2^{\frac{1}{p}-1} \right)\Big|\sum_{i=1}^{n} \lambda_i \Big| + 2^{\frac{1}{p} -1} \sum_{i=1}^n |\lambda_i| : A\in \mathcal{H}^{(\beta)}_n, \sum_{i=1}^n \lambda_i^p =1 \Big\}.$$
Let $n\geq 1$. We consider the optimization problem
$$ \inf\Big\{  \left( 1-2^{\frac{1}{p}-1} \right)\Big|\sum_{i=1}^{n} \lambda_i \Big| + 2^{\frac{1}{p} -1} \sum_{i=1}^n |\lambda_i| : A\in \mathcal{H}^{(\beta)}_n, \sum_{i=1}^n \lambda_i^p =1 \Big\}.$$
Denote for all $\lambda\in \RR^n$,
$$ \phi(\lambda) =   \left( 1-2^{\frac{1}{p}-1} \right)\Big|\sum_{i=1}^{n} \lambda_i \Big| + 2^{\frac{1}{p} -1} \sum_{i=1}^n |\lambda_i| .$$
Compactness and continuity arguments show that the infimum is achieved at some $\lambda \in \RR^n$. At the price of permuting the coordinates of $\lambda$, and taking the opposite of $\lambda$, which does not change the value of $\phi(\lambda)$, we can assume that $\lambda = \left( \lambda_1,...,\lambda_m,0,...,0\right)$, with $\lambda_1\neq 0,...,\lambda_m\neq0$ such that $\sum_{i=1}^m \lambda_i \geq 0$. Assume first that $\sum_{i=1}^m \lambda_i>0$. The multipliers rule (see \cite[Theorem 9.1]{Clarke}) yields that there is some $\gamma>0$, such that for any $i\in \{1,...,m\}$,
\begin{equation} \label{stateq} \left( 1-2^{\frac{1}{p}-1} \right) + 2^{\frac{1}{p} -1} \sg(\lambda_i) = \gamma \lambda_i^{p-1}.\end{equation}
Multiplying the above inequality by $\lambda_i$, and summing over all $i\in \{1,...,m\}$, we get
\begin{equation}\label{valuephi} \gamma = \phi(\lambda).\end{equation}
From \eqref{stateq}, we have for all $\in \{1,...,m\}$,
$$  \lambda_i = 
\begin{cases}
\gamma^{-\frac{1}{p-1}} & \text{ if } \lambda_i>0,\\
-\gamma^{-\frac{1}{p-1}} \left(2^{\frac{1}{p}} -1\right)^{\frac{1}{p-1}} & \text{ if } \lambda_i <0.
\end{cases}$$
Let $k$ denote the number of positive $\lambda_i$'s, and $l$ the number of negative $ \lambda_i$'s. As $\sum_{i=1}^m \lambda_i >0$, we have $k\geq 1$. Since $\sum_{i=1}^m \lambda_i^p = 1$, we have
$$ \gamma^{\frac{p}{p-1}} = k + l \left(2^{\frac{1}{p}} -1\right)^{\frac{p}{p-1}}\geq 1,$$
as $k\geq 1$. Thus, $\phi(\lambda) \geq1$.

Assume now that $\sum_{i=1}^m \lambda_i= 0$. Then the multipliers rule asserts that there are some $t \in [-1,1]$ and $\gamma$, such that $(t,\gamma) \neq (0,0)$, and for all $i\in \{1,...,m\}$,
$$ \Big( 1 -2^{\frac{1}{p}-1} \Big)t + 2^{\frac{1}{p}-1} \sg(\lambda_i) = \gamma \lambda_i^{p-1}.$$
At the price of changing $\lambda$ to $-\lambda$, we can assume $t \geq0$. As in the previous case, multiplying by $\lambda_i$ in the above equation and summing over $i$, yields $\phi(\lambda)=\gamma$. Note that since $\phi(1,0,...,0) =1$, we can assume $\gamma \leq 1$. 
We can write for any $i\in \{1,...,m\}$,
$$ \lambda_i =\begin{cases}
-\gamma^{-\frac{1}{p-1}} \Big( 2^{\frac{1}{p}-1} - \Big( 1-2^{\frac{1}{p}-1}\Big)t \Big)^{\frac{1}{p-1}} & \text{ if } \lambda_i<0,\\
\gamma^{-\frac{1}{p-1}} \Big( 2^{\frac{1}{p}-1} +\Big( 1-2^{\frac{1}{p}-1}\Big)t \Big)^{\frac{1}{p-1}} & \text{ if } \lambda_i>0.
\end{cases}$$
Let $k$ denotes the number of positive coordinates of $\lambda$, and by $l$ the number of negative coordinates. As $\sum_{i=1}^m \lambda_i = 0$, we have $k,l\geq 1$, and 
$$  k\Big( 2^{\frac{1}{p}-1} +\Big( 1-2^{\frac{1}{p}-1}\Big)t \Big)^{\frac{1}{p-1}}= l \Big( 2^{\frac{1}{p}-1} - \Big( 1-2^{\frac{1}{p}-1}\Big)t \Big)^{\frac{1}{p-1}}.$$
But then,
$$\phi(\lambda) =2^{\frac{1}{p}}k\gamma^{-\frac{1}{p-1}}\Big( 2^{\frac{1}{p}-1} +\Big( 1-2^{\frac{1}{p}-1}\Big)t \Big)^{\frac{1}{p-1}}\geq 2^{\frac{1}{p}} 2^{-\frac{1}{p}} =1,$$
as $\gamma\leq 1$. 
As $\phi(1,0,..,0) = 1$, we can conclude
$$ \inf \left\{ \phi(\lambda) : || \lambda ||_p=1 \right\}  =1.$$
This yields,
$$c_p \geq \min\left( b,2^{-\frac{\alpha}{p}}a\right),$$
in the case where $p$ is even. 

\end{proof}

%
%

\newpage

\bibliographystyle{plain}
\bibliography{traces.bib}{}

\end{document}